\documentclass[a4paper,10pt]{amsart}

\usepackage[latin1]{inputenc}
\usepackage[T1]{fontenc}

\usepackage{amsmath, amsthm, amssymb}
\usepackage{amscd}

\usepackage[super]{nth}
\usepackage{hyperref}

 \usepackage[all]{xy}
\usepackage{dsfont}
\usepackage{mathtools}

\usepackage[usenames,dvipsnames]{pstricks}
\usepackage{epsfig}
\usepackage{pst-grad} %
\usepackage{pst-plot} %
\usepackage{pst-solides3d}
\usepackage{subfig}
\usepackage{pstricks-add}

\newtheorem{thm}{Theorem}[section]
\newtheorem{cor}[thm]{Corollary}
\newtheorem{prop}[thm]{Proposition}
\newtheorem{defin}[thm]{Definition}
\newtheorem{lem}[thm]{Lemma}
\newtheorem{claim}[thm]{Claim}
\newtheorem*{thm*}{Theorem}
\newtheorem*{prop*}{Proposition}
\newtheorem*{defin*}{Definition}
\newtheorem*{lem*}{Lemma}
\newtheorem*{claim*}{Claim}

\newtheorem{thmintro}{Theorem}

\theoremstyle{remark}
\newtheorem{rem}[thm]{Remark}
\newtheorem*{rem*}{Remark}

\usepackage{color}

\usepackage{style_avec_bruno}

\title{Eigenvalues control for a Finsler--Laplace operator}

\author{Thomas Barthelm\'e }
\address{Insitut de Math\'ematiques, Universit\'e de Neuch\^atel}
\email{thomas.barthelme@unine.ch}
\urladdr{https://sites.google.com/site/thomasbarthelme/}
\thanks{The first author was supported by the FNS grant no.\ 20-137696/1}

\author{Bruno Colbois}
\address{Insitut de Math\'ematiques, Universit\'e de Neuch\^atel}
\email{bruno.colbois@unine.ch}

\begin{document}

\begin{abstract}
 Using the definition of a Finsler--Laplacian given by the first author, we show that two bi-Lipschitz Finsler metrics have a controlled spectrum. We deduce from that several generalizations of Riemannian results. In particular, we show that the spectrum on Finsler surfaces is controlled above by a constant depending on the topology of the surface and on the quasireversibility constant of the metric. In contrast to Riemannian geometry, we then give examples of highly non-reversible metrics on surfaces with arbitrarily large first eigenvalue.
\end{abstract}

\maketitle

\section{Introduction}

There has been several different generalizations of the Laplace--Beltrami operator to the Finslerian context \cite{BaoLackey,Shen:non-linear_Laplacian,Cen:mean-value_laplacian}. However their study seems in general fairly hard. For instance, to our knowledge, the only known result about eigenvalues is given by Munteanu \cite{Munteanu:eigen_estimates} in the case of Randers spaces. Following an idea of Patrick Foulon, the first author introduced in \cite{moi:natural_finsler_laplace} another generalization of the Laplace operator which seems more approachable and that we study in this article.

There is a very rich literature on the study of the spectrum of the Laplace--Beltrami operator, in particular on finding bounds on eigenvalues or constructing metrics with either large or small eigenvalues. First of all, it is immediate that when you scale the metric, the eigenvalues are multiplied by the inverse of the square of the scaling, and this fact stays true in the Finslerian context. Hence, when talking about large or small eigenvalues, the volume needs to be fixed in order to have a non-trivial question.

One efficient way to obtain coarse information about the Laplace--Beltrami spectrum is by comparing bi-Lipschitz metrics. Indeed, Dodziuk \cite{Dodziuk:eigen_on_forms} showed that a control of the ratio of two Riemannian metrics gives a control of their respective spectrum. Dodziuk proved it for the Hodge--Laplacian, but in the case of functions, the proof is quite straightforward: the energy associated with the Laplace--Beltrami operator does not depend on any derivatives of the metric, hence the ratio of the energy of two bi-Lipschitz equivalent metric is controlled and the Min-Max principle immediately gives a control of the ratio of the spectra. So two Riemannian drums that are roughly the same shape sounds roughly the same.

The main result of this article is that this still holds for Finslerian drums:
\begin{thmintro} \label{thmintro_spectrum_control}
Let $F$ and $F_0$ be two Finsler metrics on a compact $n$-manifold $M$.
Suppose that there exists $C > 1$ such that, for any $(x,v) \in TM$,
\begin{equation*}
 C^{-1} \leq \frac{F(x,v)}{F_{0}(x,v)} \leq C.
\end{equation*}
 Let $C_1$ and $C_2$ be the quasireversibility constant of $F$ and $F_0$ respectively. Then, there exists a constant $K \geq 1$, depending on $C$, $C_1$, $C_2$ and $n$, such that, for any $k \in \N^{\ast}$,
\begin{equation*}
 C^{-K} \leq \frac{\lambda_k(M,F)}{\lambda_k(M,F_0)} \leq C^{K}.
\end{equation*}
\end{thmintro}
Note that contrarily to the Riemannian case, this result is not trivial as derivatives of the Finsler metric does appear in the energy associated with the Finsler--Laplacian (See Section \ref{subsec_energy_and_spectrum} for the definition of the energy). We do however manage to prove that the Energy of two bi-Lipschitz metrics are nevertheless controlled and conclude again via the Min-Max principle.

We do not ask for our Finsler metrics to be reversible, i.e., to be such that the norm of a vector equals the norm of its opposite. So, saying that $F$ has a quasireversibility constant of $C_1$ means that the ratio of the Finsler norms of a vector and its opposite is controlled by $C_1^{-1}$ and $C_1$.

V.\ Matveev and M.\ Troyanov \cite{MatveevTroyanov} showed that to any Finsler metric $F$, we can associate a smooth Riemannian metric, called the Binet--Legendre metric, which is bi-Lipschitz equivalent to $F$ and such that the Lipschitz constant depends only on the dimension of the manifold and on the quasi-reversibility constant of $F$.

Thanks to the Binet--Legendre metric and Theorem \ref{thmintro_spectrum_control}, we can extend a lot of the Riemannian results to the Finslerian context. We did not try to give an exhaustive list of such results, as it should be clear to the reader that any coarse control of the spectrum of a class of Riemannian metric will yield a coarse control of the Finslerian spectrum. However, we do give the following applications, which were famous Riemannian problems:
\begin{thmintro} \label{thmintro_control_on_surfaces}
  Let $\Sigma$ be a surface of genus $\delta$. Let $F$ be a $C_1$-quasireversible Finsler metric on $\Sigma$. There exists a constant $K\geq 1$ depending only on $C_1$ such that, for any $k\in \N$,
\begin{equation*}
 \lambda_k(\Sigma,F) \left(\vol(\Sigma,F) \right) \leq (2C_1)^K (1+\delta) k.
\end{equation*}
\end{thmintro}

\begin{thmintro} \label{thmintro_conformal_control}
Let $M$ be a compact $n$-manifold and $F$ a Finsler metric on $M$. There exists a constant $C_n\left([F] \right)$, depending only on the dimension $n$ and the conformal class of $F$, such that, for any $k\in \N$, 
\begin{equation*}
 \lambda_k(M,F) \left(\vol\left(M,F \right) \right)^{2/n} \leq C_n\left( [F]\right) k^{2/n}.
\end{equation*}
\end{thmintro}

\begin{thmintro} \label{thmintro_covering_tower}
 Let $M$ be a compact $n$-manifold and $F$ a Finsler metric on $M$. Let $\{M_i\}_{i\geq 1}$ be a family of finite-sheeted covering spaces of $M$ with their induced Finsler metric. Let $\Gamma_i$ be the Schreier graph of the subgroup $\pi_1(M_i)$ of $\pi_1(M)$. Then, there exists a constant $C \geq 1$ depending on $n$ and $F$ such that, for all $k < \left|\Gamma_i \right| $
\begin{equation*}
 C^{-1} \lambda_k \left( \Gamma_i \right)  \leq \lambda_k\left(M_i \right) \leq C \lambda_k \left( \Gamma_i \right).
\end{equation*}
In particular, for all $k$
\begin{equation*}
 \lambda_k\left(M_i \right) \rightarrow 0 \; \text{when } i\rightarrow +\infty \Leftrightarrow \lambda_k \left( \Gamma_i \right) \rightarrow 0 \; \text{when } i\rightarrow +\infty .
\end{equation*}
\end{thmintro}

The Theorem \ref{thmintro_control_on_surfaces} leads to a very natural question: can we get rid of the dependency on the quasireversibility constant in the bound for the eigenvalues? Or in other words, can we prove that the first eigenvalue is bounded for \emph{any} Finsler metric on a surface? We show that this is never the case:
\begin{thmintro}\label{thmintro_big_eigenvalues}
 For any surface $\Sigma$ and any $M>0$, there exists a Randers metric $F = \sqrt{g} + \rho$ such that 
\begin{equation*}
 \lambda_1(F) \vol\left(\Sigma, F \right) \geq M.
\end{equation*}
\end{thmintro}
 In our opinion, the interest of this result is three-fold. First it exhibits a behavior that is impossible for Riemannian metric on surfaces, i.e., large eigenvalues. Second, it adds one item to the list of surprise that we can get when considering non-reversible Finsler metrics instead of reversible ones. Finally, it shows that this Finsler--Laplacian can hear when a drum is non-reversible.

Note that for manifolds of dimension $3$ and greater, there always exists large eigenvalues already in the Riemannian setting \cite{ColboisDodziuk}.

\subsection*{Structure of this paper}
In Section \ref{sec_definitions}, we introduce the definitions and basic results that we will need. The main references for this section are \cite{moi:these,moi:natural_finsler_laplace}.

In Section \ref{sec_control_of_energy}, we prove that we can control the ratio of energy of two bi-Lipschitz equivalent metrics (Theorem \ref{thm_energy_control}) and deduce Theorem \ref{thmintro_spectrum_control} (Corollary \ref{cor_control_spectrum}). We then obtain the Theorems \ref{thmintro_control_on_surfaces}, \ref{thmintro_conformal_control} and \ref{thmintro_covering_tower} in Section \ref{sec_applications}.

Finally, in Section \ref{sec_big_eigenvalues}, we construct Randers surfaces with arbitrarily large first eigenvalue. We do it first on the torus, as the construction is straightforward, and then adapt the construction to any surfaces.

\section{Definitions} \label{sec_definitions}

In this section, we will give the definition of the Finsler--Laplacian we use as well as the main results we will need, the reader can consult \cite{moi:these,moi:natural_finsler_laplace} for proofs.

 First, let us state the definition of Finsler metric we will be using.
\begin{defin} \label{def:finsler_metric}
Let $M$ be a manifold. A (smooth) Finsler metric on $M$\ is a continuous function ${F \colon TM \rightarrow \R^+}$ that is:
\begin{enumerate}
  \item $C^{\infty}$\ except on the zero section,
  \item positively homogeneous, i.e., $F(x,\lambda v)=\lambda F(x,v)$\ for any $\lambda>0$,
  \item positive-definite, i.e., $F(x,v)\geq0$\ with equality iff $v=0$,
  \item strongly convex, i.e., $ \left(\dfrac{\partial^2 F^2}{\partial v_i \partial v_j}\right)_{i,j}$ is positive-definite.
 \end{enumerate}
\end{defin}
A Finsler metric is said to be \emph{reversible} if $F(x, -v ) = F(x,v)$\ for any $(x,v)\in TM$. We say that $C_1$ is the \emph{quasireversibility constant} of a Finsler metric $F$ if 
\begin{equation*}
 C_1 =  \sup \left\{ F(x,-v) \mid (x,v) \in TM, \; \text{such that } F(x,v) =1 \right\}.
\end{equation*}

The Finsler--Laplacian is defined using Foulon's formalism \cite{Fou:EquaDiff} that we quickly recall.
Let $HM$\ be the homogenized bundle, i.e., $HM := \left(TM \smallsetminus \{\text{zero section}\} \right) / \R^+$. Let $\pi \colon HM \rightarrow M$\ be the canonical projection and $VHM = \Ker d\pi \subset THM$ the vertical bundle.

The cornerstone of Foulon's formalism is the Hilbert form $A$ associated to the Finsler metric. The Hilbert form is a $1$-form on $HM$ defined, for $(x,\xi) \in HM$, $Z \in T_{(x,\xi)}HM$, and $v \in T_xM$ such that $r(x,v) = (x,\xi)$, where $r \colon TM\smallsetminus \{\text{zero section}\}  \rightarrow HM$, by
\begin{equation}
 A_{(x,\xi)}(Z) := \lim_{\eps \rightarrow 0} \frac{F\left(x, v + \eps d\pi(Z) \right) - F\left( x,v \right)}{\eps}.
\end{equation}
The Hilbert form contains all the necessary information about the dynamics of the Finsler metric:

\begin{thm} [Hilbert]
 The form $A$ is a contact form, i.e., if $n$ is the dimension of $M$, $\ada$ is a volume form on $HM$.
 Let $X \colon HM \rightarrow THM$ be the Reeb field of $A$, i.e., the only solution of 
\begin{equation}
\label{eq:Reeb_field}
\left\{ 
\begin{aligned}
  A(X) &= 1 \\
 i_X dA &= 0  \, .
 \end{aligned}
\right.
\end{equation}
The vector field $X$ generates the geodesic flow for $F$.
\end{thm}

We can now define the Finsler--Laplacian. First we split the canonical volume $\ada$ into a volume form on the manifold $M$ and an angle form:
\begin{prop}
\label{prop:construction}
 There exists a unique volume form $\Omega^F$\ on $M$\ and an $(n-1)$-form $\alpha^F$\ on $HM$, never zero on $VHM$, such that
\begin{equation}
\label{eq:alpha_wedge_omega}
  \alpha^{F} \wedge \pi^{\ast}\Omega^F =  A\wedge dA^{n-1}, 
\end{equation}
and, for all $x\in M$, 
\begin{equation}
\label{eq:longueur_fibre}
 \int_{H_xM} \alpha^F =  \voleucl(\S^{n-1})\, .
\end{equation}
\end{prop}

\begin{rem}
 The volume form $((n-1)!)^{-1} \Omega^F$ is the Holmes--Thompson volume form.
\end{rem}

The Finsler--Laplacian of a function is then obtained as an average with respect to $\alpha^F$ of the second derivatives in every directions:
\begin{defin}
\label{def:delta}
 For $f \in C^2(M)$, the operator $\Delta^F$\ is defined by, for any $x \in M$,
 $$
 \Delta^F f (x) = \frac{n}{\voleucl \left(\mathbb{S}^{n-1}\right) }\int_{H_xM} L_X ^2 (\pi^{\ast} f ) \alpha^F,
 $$
where $L_X$ denotes the Lie derivative of $X$.
\end{defin}

The constant in front of the operator is there in order to get back the usual Laplace--Beltrami operator when $F$ is Riemannian.

\subsection{Energy and spectrum} \label{subsec_energy_and_spectrum}

The Finsler--Laplacian has a naturally associated energy functional defined, for $f \in H^1(M)$, by
\begin{equation}
 E^F(f) := \frac{n}{\voleucl \left(\S^{n-1}\right) } \int_{HM} \left|L_X\left(\pi^{\ast}f \right)\right|^2 \ada. \label{eq_energy_finsler}
\end{equation}
The \emph{Rayleigh quotient} for $F$ is then defined by
\begin{equation}
 R^F(f) := \frac{E^F(f)}{\int_M f^2\, \Omega^F}.
\end{equation}

On compact manifolds, the spectrum of the Finsler--Laplacian is discrete and can be obtained, as in the Riemannian case, via the Min-Max principle. For simplicity, we state it for closed manifold:
\begin{thm}[Min-Max principle] \label{thm_min_max}
 Let $M$ be a closed manifold and $F$ a Finsler metric on $M$.
 Let $\lambda_k$ be the $k$th eigenvalue (counted with multiplicity) of $-\Delta^F$, then
\begin{equation*}
 \lambda_k = \inf_{V_k} \sup \left\{ R^F(u) \mid u \in V_k \right\}
\end{equation*}
where $V_k$ runs over all the $k$-dimensional subspaces of $H^1(M)$.
\end{thm}

\subsection{Legendre transform and dual metrics}

Finsler geometry can also be studied via the cotangent bundle, this dual point of view can sometimes be of tremendous help and will be used in this article. What we present here is fairly well known but the reader can refer for instance to \cite{moi:these} for precisions.

\begin{defin}
  Let $F$ be a Finsler metric on a manifold $M$. The dual Finsler metric ${F^{\ast} \colon T^{\ast}M \rightarrow \R}$ is defined, for $(x,p) \in T^{\ast}M$, by
\begin{equation*} %
F^{\ast}(x,p) = \sup \lbrace p(v) \mid v\in T_xM \; \text{\rm such that } F(x,v)=1 \rbrace.
\end{equation*}
\end{defin}

The Legendre transform allows one to switch from the tangent bundle to the cotangent bundle.
\begin{defin}
 The \emph{Legendre transform} $\L_F : TM \rightarrow T^{\ast}M$ associated with $F$ is defined by $L_F(x,0) = (x,0)$ and, for $(x,v) \in \mathring{T}M$ and $u \in T_xM$, 
\begin{equation*}
 \L_F ( x,v) (u) := \frac{1}{2} \left. \frac{d}{dt} F^2(x,v + tu)\right|_{t=0}.
\end{equation*}
\end{defin}
As $F^2$ is $2$-homogeneous, we have that $\L_F$ is $1$-homogeneous, so we can project $\L_F$ to the homogenized bundles. Set $H^{\ast}M := \mathring{T}^{\ast}M / \R^+_{\ast}$ and write $\ell_F\colon HM \rightarrow H^{\ast}M$ for the projection. Considering directly $\ell_F$, instead of $\L_F$, can be quite helpful sometimes.

The Legendre transform $\L_F$ is a diffeomorphism and the following diagram commutes (see for instance \cite{moi:these}):
$$
\xymatrix{
    & \Tzero^{\ast}M \ar[r]^{\hat{r}} \ar[ld]_{\hat{p}}  &   H^{\ast}M \ar[rd]^{\hat{\pi}} & \\
M   &       &       &   M  \\
    & \Tzero M \ar[uu]^{\L_F} \ar[r]_r  \ar[ul]^{p} & HM \ar[uu]_{\ell_F} \ar[ur]_{\pi} & 
}
$$
For strongly convex smooth Finsler metrics, the Legendre transform can also be described using convex geometry. The Legendre transform associated with a convex $\mathcal{C} \subset \R^n$ sends a point $x$ of $\mathcal{C}$ to the hyperplane supporting $\mathcal{C}$ at $x$, or equivalently, to the linear map $p \in \left(\R^n\right)^{\ast}$ such that $p(x)=1$ and $\ker p$ is parallel to the supporting hyperplane.

The following result will be very important for us, it is due to P.\ Foulon but was never published, we provide the proof (taken from \cite{moi:these}) below.
\begin{thm}[Foulon \cite{Fou:perso}] \label{thm_uniquely_contact}
Any Finsler metric on $M$ defines the same contact structure on $H^{\ast}M$, i.e., if $F$ is a Finsler metric on $M$ and ${B = \left(\ell_F^{-1} \right)^{\ast} A}$, the distribution $\ker B \subset TH^{\ast}M$ is independent of $F$.\\
Furthermore, if we denote by $\lambda$ the Liouville $1$-form on $T^{\ast}M$, we have 
\begin{equation} \label{eq:B_equal_louiville_over_F_star}
 \hat{r}^{\ast} B = \frac{\lambda}{F^{\ast}},
\end{equation}
and
\begin{equation} \label{eq:bdb_and_louiville}
 \hat{r}^{\ast} B\wedge dB^{n-1} = \frac{\lambda \wedge d\lambda^{n-1} }{(F^{\ast})^n}. 
\end{equation}
\end{thm}

\begin{proof} 
We will start by showing Equation \eqref{eq:B_equal_louiville_over_F_star}. First recall the definition of the Liouville form: for any $\in T^{\ast}M$, $\lambda_{p} = p \circ d\hat{p}_{|M}$, where $\hat{p}_{|M} \colon T^{\ast}M \rightarrow M$ is the base point projection. In order to show that $\hat{r}^{\ast} B = \frac{\lambda}{F^{\ast}}$, we will prove that their pull-back by $\L_F$ coincides.

On one hand, as $\hat{r} \circ \L_F = r \circ \ell_F$, we have
\begin{equation*}
\L_F^{\ast} \hat{r}^{\ast} B = r^{\ast}  \ell_F^{\ast} B = r^{\ast} A = d_vF \, ,
\end{equation*}
and on the other hand,
\begin{equation*}
 \L_F^{\ast} \left( \frac{\lambda}{F^{\ast}} \right) = \frac{\L_F^{\ast} \lambda }{ F^{\ast} \circ \L_F} = \frac{\L_F^{\ast} \lambda }{F} \, .
\end{equation*}
Now, let us compute $\L_F^{\ast} \lambda $: for $(x,v) \in TM$ and $Z \in T_{(x,v)} TM$,
\begin{align*}
\left( \L_F^{\ast} \lambda\right)_{(x,v)} (Z) &= \lambda_{\L_F(x,v)} \left( d\L_F (Z) \right) \\
	&= \L_F(x,v) \circ d\hat{p}_{|M} \circ d\L_F (Z) \\
	&= \L_F(x,v) \circ dp_{|M} (Z) \\
	&= \frac{1}{2} \frac{d}{dt} F^2\left(x, v + t dp_{|M} (Z) \right) \\
	&= F(x,v) d_vF_{(x,v)}(Z)\, .
\end{align*}
And we proved Equation \eqref{eq:B_equal_louiville_over_F_star}. Once we have that, the uniqueness of the contact structure is trivial.
 
For the last equality, we have
\begin{equation*}
 \hat{r}^{\ast} dB = d \hat{r}^{\ast} B =  \frac{d\lambda}{F^{\ast}} - \frac{\lambda \wedge dF^{\ast} }{(F^{\ast})^2}\, .
\end{equation*}
Therefore $\hat{r}^{\ast} dB^{n-1} =  \left(\frac{d\lambda}{F^{\ast}}\right)^{n-1} + \lambda\wedge\left( \text{Something} \right)$, so
\begin{equation*}
 \hat{r}^{\ast} B\wedge dB^{n-1}=  \frac{\lambda \wedge d\lambda^{n-1}}{(F^{\ast})^{n}} + \lambda \wedge \lambda\wedge\left( \text{Something} \right) = \frac{\lambda \wedge d\lambda^{n-1}}{(F^{\ast})^{n}}\, . \qedhere
\end{equation*}
\end{proof}

This Theorem allows us to deduce expressions for the volume and angle of one Finsler metric with respect to another one:
\begin{lem} \label{lem_alpha_and_omega}
 Let $F$ and $F_0$ be two Finsler metrics on $M$. Let $\mu \colon M \rightarrow \R$ defined by
\begin{equation*}
 \mu (x) :=  \left(\voleucl \left(\S^{n-1} \right) \right)^{-1} \int_{H_x^{\ast} M} \left(\frac{F_0^{\ast}}{F^{\ast}} \right)^n \beta^{F_0},
\end{equation*}
where $\beta^{F_0} = \left(\ell_{F_0}^{-1}\right)^{\ast} \alpha^{F_0}$, and $ F_0^{\ast}/F^{\ast}$ is seen as a function on $H^{\ast}M$.
Then we have
\begin{align*}
 \Omega^F &= \mu(x) \Omega^{F_0} \\
 \alpha^F &= \left(\ell_{F_0}^{-1} \circ \ell_F \right)^{\ast} \left[ \mu^{-1} \left(\frac{F_0^{\ast}}{F^{\ast}} \circ \ell_{F_0}^{-1} \right)^n \alpha^{F_0} \right].
\end{align*}
\end{lem}

\begin{proof}
 Let $F$ and $F_0$ be two Finsler metrics on $M$. Let $A$ and $A_0$ be the Hilbert forms of respectively $F$ and $F_0$, and let $B = \left(\ell_F^{-1}\right)^{\ast} A$ and $B_0 = \left(\ell_{F_0}^{-1}\right)^{\ast} A_0$.
 By Theorem \ref{thm_uniquely_contact}, 
\[
 B\wedge dB^{n-1} = \left(\frac{F_0^{\ast}}{F^{\ast}}\right)^n B_0 \wedge dB_0^{n-1},
\]
where $F_0^{\ast}/F^{\ast}$ is considered as a function from $H^{\ast}M$ to $\R$. Now, if $\beta^F = \left(\ell_F^{-1}\right)^{\ast}\alpha^F$,  $\beta^{F_0} = \left(\ell_{F_0}^{-1}\right)^{\ast} \alpha^{F_0}$ and $\hat{\pi} \colon H^{\ast}M \rightarrow M$ is the natural projection, we have, by Proposition \ref{prop:construction},
\begin{equation*}
 \beta^F \wedge \hat{\pi}^{\ast} \Omega^F =  B\wedge dB^{n-1} = \left(\frac{F_0^{\ast}}{F^{\ast}}\right)^n B_0 \wedge dB_0^{n-1} = \left(\frac{F_0^{\ast}}{F^{\ast}}\right)^n \beta^{F_0} \wedge \hat{\pi}^{\ast} \Omega^{F_0}.
\end{equation*}
As $\Omega^F$ is the unique volume form such that $\beta^F \wedge \hat{\pi}^{\ast} \Omega^F =  B\wedge dB^{n-1}$ and $\int_{H_x^{\ast}M} \beta^F = \voleucl \left( \S^{n-1} \right)$, we obtain
\begin{equation*}
 \Omega^F =  \frac{\int_{H_x^{\ast} M} \left(\frac{F_0^{\ast}}{F^{\ast}} \right)^n \beta^{F_0}}{\voleucl \left(\S^{n-1} \right)} \; \Omega^{F_0},
\end{equation*}
and 
\begin{equation*}
 \beta^F = \frac{\voleucl \left(\S^{n-1} \right)}{\int_{H_x^{\ast} M} \left(\frac{F_0^{\ast}}{F^{\ast}} \right)^n \beta^{F_0}}   \left(\frac{F_0^{\ast}}{F^{\ast}}\right)^n \beta^{F_0} .
\end{equation*}
Writing $\beta^F = \left(\ell_F^{-1}\right)^{\ast}\alpha^F$ and $\beta^{F_0} = \left(\ell_{F_0}^{-1}\right)^{\ast} \alpha^{F_0}$ gives the result.
\end{proof}

\section{Bi-Lipschitz control of the energy} \label{sec_control_of_energy}

The goal of this section is to prove the following theorem:

\begin{thm} \label{thm_energy_control}
 Let $F$ and $F_0$ be two Finsler metrics on a $n$-manifold $M$.
Suppose that there exists $C >1$ such that, for any $(x,v) \in TM$,
\begin{equation*}
 C^{-1} \leq \frac{F(x,v)}{F_{0}(x,v)} \leq C.
\end{equation*}
 Let $C_1$ and $C_2$ be the quasireversibility constants of $F$ and $F_0$ respectively. Then, there exists a constant $K\geq 1$, depending on $C$, $C_1$, $C_2$ and $n$, such that, for any $f \in H^1(M)$,
\begin{equation*}
 C^{-K} \leq \frac{E^F(f)}{E^{F_{0}}(f)} \leq C^K.
\end{equation*}
\end{thm}
Once again, let us emphasize that even so this result seems natural, it is far from obvious because the energy does a priori depends on the derivatives of the Finsler metric.

\begin{rem}
 In the proof of the Theorem, we will show that, if $\sigma^F$ and $\sigma^{F_0}$ are the respective symbols of $\Delta^F$ and $\Delta^F_0$, then there exists a constant $K(C,C_1,C_2,n)$ such that, for any $p \in T^{\ast}M$,
\begin{equation*}
 C^{-K(C,C_1,C_2,n)} \leq \frac{\lVert p\rVert_{\sigma_F}}{\lVert p\rVert_{\sigma_{F_0}} } \leq C^{K(C,C_1,C_2,n)}.
\end{equation*}
Note also that there is a link between $C_1$, $C_2$ and $C$, for instance, given $C$ and $C_1$, we can get an upper bound on $C_2$. However, this fact is not useful for our purposes.
\end{rem}

But before starting the proof of this result, let us state its main corollary. Recall that, when $M$ is compact, the spectrum of the Laplacian is obtained from the Rayleigh quotient via the Min-Max Principle (see Theorem \ref{thm_min_max}). As our Theorem gives a control of the energy of two bi-Lipschitz metrics and the Lemma \ref{lem_control_of_m_mu} below gives a control of the volumes, we control the Rayleigh quotient. A direct consequence is that we control the spectrum:
\begin{cor} \label{cor_control_spectrum}
 Let $M$ be a compact $n$-manifold. Let $F$ and $F_0$ be two Finsler metrics on $M$ such that, for some $C >1$ and for any $(x,v) \in TM$,
\begin{equation*}
 C^{-1} \leq \frac{F(x,v)}{F_{0}(x,v)} \leq C.
\end{equation*}
  Let $\lambda_k(F)$ and $\lambda_k(F_0)$ be the $k$th eigenvalue of $-\Delta^F$ and $-\Delta^{F_0}$ respectively. Then, there exists a constant $K'\geq 1$, depending on $C$, $C_1$, $C_2$ and $n$, such that
\begin{equation*}
 C^{-K'} \leq \frac{\lambda_k(F)}{\lambda_k(F_0)} \leq C^{K'}.
\end{equation*}
\end{cor}

The proof of Theorem \ref{thm_energy_control} takes up the rest of this section.

\begin{lem} \label{lem_control_of_m_mu}
 Let $F$ and $F_0$ be two Finsler metrics on a $n$-manifold $M$.
Suppose that there exists $C >1$ such that, for any $(x,v) \in TM$,
\begin{equation*}
 C^{-1} \leq \frac{F(x,v)}{F_{0}(x,v)} \leq C.
\end{equation*}
Let $m \colon HM \rightarrow \R$ and $Y \colon HM \rightarrow VHM$ such that $X^F = mX^{F_0} +Y$. Let ${\mu \colon M \rightarrow \R}$ such that $\Omega^F = \mu \Omega^{F_0}$. Then,
\begin{align}
 C^{-1} &\leq \frac{F^{\ast}(x,v)}{F^{\ast}_{0}(x,v)} \leq C, \label{eq_dual_control}\\
 C^{-n} &\leq \mu \leq C^n \label{eq_control_volume} , \\
 C^{-1} &\leq m \leq C \label{eq_control_m} .
\end{align}
\end{lem}

\begin{proof}
Let us start by proving Equation \eqref{eq_dual_control}.\\
 Using the characterization of the dual norms as supremum, we have
\begin{equation*}
 F^{\ast}(x,p) = \sup_{v \in T_xM} \frac{p(v)}{F(x,v)} = \sup_{v \in T_xM} \frac{F_0(x,v)}{F(x,v)}\frac{p(v)}{F_0(x,v)}
\end{equation*}
and the equation follows directly.

We can now prove Equation \eqref{eq_control_volume}:\\
We saw in Lemma \ref{lem_alpha_and_omega} that 
\begin{equation*}
 \mu = \left(\voleucl \left(\S^{n-1} \right) \right)^{-1} \int_{H_xM} \left(\frac{F_0^{\ast}}{F^{\ast}} \right)^n \beta^{F_0},
\end{equation*}
so Equation \eqref{eq_control_volume} follows immediately.

For the last equation, we just have to remark that $1= A^F(X^F) = mA^F(X^{F_0})$. Now, for $(x,\xi) \in HM$, $A^F\left(X^{F_0}(x,\xi) \right) = F(x,v)$ where $v \in T_xM$ is a representative of $\xi$ such that $F_0(x,v)=1$. Hence the conclusion.
\end{proof}

The bulk of the proof of Theorem \ref{thm_energy_control} is contained in the following result
\begin{prop} \label{prop_symbol_control}
 Let $F$ be a Finsler metric on a $n$-manifold $M$, and $g_0$ a Riemannian metric on $M$ such that for some constant $C>1$, we have 
\begin{equation*}
 C^{-1} \leq \frac{F(x,v)}{\sqrt{g_{0}(x,v)}} \leq C.
\end{equation*}
Let us denote by $\sigma^F$ the symbol of the Finsler--Laplacian $\Delta^F$. There exists a constant $K\geq 1$, depending on $C$ and $n$, such that, for $p\in T^{\ast}M$
\begin{equation*}
 C^{-K} \leq \frac{\lVert p\rVert_{\sigma^F}}{\lVert p\rVert_{g_{0}^{\ast}}} \leq C^K.
\end{equation*}
\end{prop}

\begin{proof}
Let $p \in T_x^{\ast}M$ be fixed. We suppose that $\lVert p\rVert_{g_{0}^{\ast}} = 1$. Let $\phi \colon M \rightarrow \R$ be a smooth function such that $\phi(x) = 0$ and $d\phi_x = p$. Then, the norm of $p$ for the symbol metric is
\begin{equation*}
 \lVert p\rVert_{\sigma^F} = \frac{n}{\voleucl\left(\S^{n-1} \right)} \int_{H_xM} \left(L_X \pi^{\ast} \phi \right)^2 \alpha^F.
\end{equation*}
Let us write $c_n:= n \left(\voleucl\left(\S^{n-1}\right) \right)^{-1}$, and from now on, we will write $F_0:= \sqrt{g_0}$. Let $X_0$ and $X$ be the geodesic vector fields associated with $F_0$ and $F$ respectively. There exists $m \colon M \rightarrow \R$ and $Y \colon HM \rightarrow VHM$ such that $X = m X_0 + Y$, so, using Lemma \ref{lem_alpha_and_omega} and the change of variable formula, we get
\begin{align*}
 \lVert p\rVert_{\sigma^F} &= c_n \int_{H_xM} m^2 \left(L_{X_0} \pi^{\ast} \phi \right)^2 \left(\ell_{F_0}^{-1} \circ \ell_F \right)^{\ast} \left[ \mu^{-1} \left(\frac{F_0^{\ast}}{F^{\ast}} \circ \ell_{F_0}^{-1} \right)^n \alpha^{F_0} \right] \\
    &= c_n \int_{H_xM} \left(m\circ \ell_{F}^{-1} \circ \ell_{F_0} \right)^2 \left(L_{X_0} \pi^{\ast} \phi \circ \ell_{F}^{-1} \circ \ell_{F_0} \right)^2  \mu^{-1} \left(\frac{F_0^{\ast}}{F^{\ast}} \circ \ell_{F_0}^{-1} \right)^n \alpha^{F_0}.
\end{align*}
Now, using Lemma \ref{lem_control_of_m_mu}, we have that
\begin{align*}
 \lVert p\rVert_{\sigma^F} &\leq c_n C^{2n +2} \int_{H_xM} \left(L_{X_0} \pi^{\ast} \phi \circ \ell_{F}^{-1} \circ \ell_{F_0} \right)^2 \alpha^{F_0},\\
 \lVert p\rVert_{\sigma^F} &\geq c_n C^{-2n -2} \int_{H_xM} \left(L_{X_0} \pi^{\ast} \phi \circ \ell_{F}^{-1} \circ \ell_{F_0} \right)^2 \alpha^{F_0}.
\end{align*}

So our goal is to obtain a control of $\int_{H_xM} \left(L_{X_0} \pi^{\ast} \phi \circ \ell_{F}^{-1} \circ \ell_{F_0} \right)^2 \alpha^{F_0}$, depending on $C$ and $n$.

Finding an upper bound is easy. Indeed, by definition of $\phi$, for any $(x,\xi) \in H_xM$, $L_{X_0}\pi^{\ast} \phi (x, \xi) = p(v)$ where $v$ is a representative of $\xi$ such that $F_0(x,v)=1$. As we supposed that $\lVert p \rVert_{g_0^{\ast}} = \sup \lbrace p(u) \mid u\in T_xM, F_0(x,u)=1 \rbrace$, we have, for any $(x,\xi) \in H_xM$, $\left| L_{X_0}\pi^{\ast} \phi (x, \xi) \right| \leq 1$. Hence,
\begin{equation*}
 \lVert p\rVert_{\sigma^F} \leq c_n C^{2n +2} \int_{H_xM} \alpha^{F_0} = n C^{2n +2}.
\end{equation*}

The goal now is to show that, for any Finsler metric $F$ bi-Lipschitz equivalent to $F_0 = \sqrt{g_0}$, the integral $\int_{H_xM} \left(L_{X_0} \pi^{\ast} \phi \circ \ell_{F}^{-1} \circ \ell_{F_0} \right)^2 \alpha^{F_0}$ cannot be too small. To achieve that, we will reduce this problem to a problem of convex geometry in $\R^n$.

Let $(x_1,\dots,x_n)$ be normal coordinates for $g_0$ at $x$ such that $p = dx_n$. Let $(v_1, \dots, v_n)$ be the associated coordinates on $T_xM$ and $(\theta_1, \dots, \theta_{n-1})$ the spherical coordinates on $H_xM$ given by
\begin{align*}
 \sin \theta_{n-1} \dots \sin \theta_{2} \cos \theta_1 &= \frac{v_1}{\sqrt{\sum v_i^2}} \\
 \sin \theta_{n-1} \dots \sin \theta_{2} \sin \theta_1 &= \frac{v_2}{\sqrt{\sum v_i^2}} \\
  &\vdots \\
 \cos \theta_{n-1} &= \frac{v_n}{\sqrt{\sum v_i^2}}.
\end{align*}

By our choice of coordinates, we have that
\begin{equation*}
 L_{X_0}\pi^{\ast} \phi = \cos \theta_{n-1}.
\end{equation*}
Let us abuse notations and write $\ell_{F}^{-1} \circ \ell_{F_0}(\theta_i)$ for the $\theta_i$ coordinate of $\ell_{F}^{-1} \circ \ell_{F_0}(\theta_1, \dots, \theta_{n-1})$.

\begin{claim}
For $\eps>0$ small enough, i.e., such that $\sin^2 \eps \leq 2 C^6 - \sqrt{4 C^{12}-1}$,  and for any Finsler metric $F$ which is $C$-bi-Lipschitz equivalent to $F_0=\sqrt{g_0}$, if $(\theta_1, \dots, \theta_{n-1})\in H_xM$ is such that
\begin{equation*}
 \cos \theta_{n-1} \geq \cos \eps,
\end{equation*}
 then
\begin{equation*}
 \left| \cos\left( \ell_{F}^{-1} \circ \ell_{F_0}(\theta_{n-1}) \right)\right| \geq \frac{C^{-2} \cos \eps}{2}.
\end{equation*}
\end{claim}

\begin{proof}
To prove the Claim, we need to describe the Legendre transform. The function $\ell_{F}^{-1} \circ \ell_{F_0}$ takes a point $\xi \in H_xM$ to a point $\xi'\in H_xM$ such that the tangent to the unit sphere of $F$ at $v'\in T_xM$, where $v'$ is a representative of $\xi'$, is parallel to the tangent to the unit sphere of $F_0$ at $v\in T_xM$, where $v$ is a representative of $\xi$.

Let $\eps>0$ and $(\theta_1, \dots, \theta_{n-1})\in H_xM$ is such that $\cos \theta_{n-1} \geq \cos \eps$.
Let $\mathcal{P} \subset T_xM$ be the plan through $(\theta_1, \dots, \theta_{n-1})$ and the $v_n$-axis. Let $\mathcal{H}$ be the hyperplane in $T_xM$ tangent at $(\theta_1, \dots, \theta_{n-1})$ to the (Euclidean) sphere $S(C^{-1})$ of center $0$ and radius $C^{-1}$. Let $\mathcal{H}_0$ be the hyperplane in $T_xM$ determined by the equation $v_n = \left(C^{-1}/2 \right) \cos \eps$. Finally, let $X$ be the distance between $0$ and $A:= \mathcal{H}_0 \cap \mathcal{H} \cap \mathcal{P}$, and $\psi$ the angle between the $v_n$-axis and the line $(0A)$ (see Figure \ref{fig_dessins_pr_preuve}).

\begin{figure}[h!]
 \centering
  \subfloat{\label{fig_dessin_1} \scalebox{1} { \begin{pspicture}(-0.5,-1.5)(5.5,7)
%
%
\rput(-0.2,-0.2){$0$}
\psarc[linewidth=0.04,linestyle=dashed](0,0){4}{-10}{90}
\rput(3.8,-1){$S\left(C^{-1} \right)$}
\psline[linewidth=0.04]{->}(0,0)(6,0)
\rput(6.2,-0.2){$v_n$}
\psline[linewidth=0.04](0,0)(4.531538935,2.113091309)
\psarc[linewidth=0.04]{->}(0,0){1}{0}{25}
\rput(1.4,0.2){$\eps$}
\psline[linewidth=0.04](1.812615574,-0.9)(1.812615574,7)
\put(1.9,-1){$\mathcal{H}_0 \cap \mathcal{P}$}
\psline[linewidth=0.04](5,-1.25772827)(1.5,6.248045952)
\put(5,-1.2){$\mathcal{H}\cap \mathcal{P}$}
\psline[linewidth=0.04](0,0)(1.812615574,5.57763969)
\psarc[linewidth=0.04, linecolor=red]{->}(0,0){0.8}{0}{71}
\rput(0.8,0.7){$\color{red}\psi$}
\rput(1,-0.4){$\frac{\cos \eps}{2C}$}
\rput(0.6, 2.6){$X$}
\put(3.9,1.6){$\left(  \theta_1, \dots, \theta_{n-2}, \eps\right)$}
\rput(2,5.7){$A$}
\end{pspicture} }} \quad
  \subfloat{\label{fig_dessin_2} \scalebox{1} { \begin{pspicture}(0,-4.5)(9,4.5)
%
\pscircle[linewidth=0.04,dimen=outer](3.6,-0.85){1.2}
\pscircle[linewidth=0.04,dimen=outer](3.6,-0.85){3.6}
\psbezier[linewidth=0.04,linecolor=red](2.98,2.03)(3.3769023,2.9241416)(3.6,2.75)(4.0,1.95)(4.4,1.15)(4.9,-0.13)(4.98,-1.23)(5.06,-2.33)(4.74,-3.11)(4.44,-3.57)(4.14,-4.03)(3.6,-4.45)(3.2,-4.05)(2.8,-3.65)(1.9356114,-2.2418473)(2.0,-1.25)(2.0643885,-0.25815263)(2.5830975,1.1358583)(2.98,2.03)
\psline[linewidth=0.04cm,linestyle=dashed,arrowsize=0.053cm 2.0,arrowlength=1.4,arrowinset=0.4]{->}(3.6,-0.85)(9.2,-0.85)
\psline[linewidth=0.04cm](2.78,4.45)(4.94,0.03)
\psline[linewidth=0.04cm](5.56,3.33)(7.72,-1.09)
\psline[linewidth=0.04cm,linestyle=dashed](3.6,-0.85)(8.06,1.41)
\psline[linewidth=0.04cm,linestyle=dashed](3.6,-0.85)(4.16,3.47)
\psarc[linewidth=0.04cm]{->}(3.6, -0.85){2}{0}{27}
\psarc[linewidth=0.04cm,linecolor=red]{->}(3.6, -0.85){0.7}{0}{82}
\rput(4.3,-0.2){$\color{red}\psi$}
\rput(5.8,-0.4){$\eps$}
\rput(3.6,-2.3){$S(C^{-1})$}
\rput(7,-3){$S(C)$}
\rput(5.4,-2.8){$\color{red}B^F(1)$}
\rput(9.3,-1.1){$v_n$}
\put(6.8,1.5){$(\theta_1, \dots , \theta_{n-2}, \eps)$}
\put(3.4,3.7){$\ell_{F}^{-1} \circ \ell_{F_0}(\theta_1, \dots ,\theta_{n-2}, \eps)$}
\end{pspicture} }}
 \caption{} \label{fig_dessins_pr_preuve}
\end{figure}

If $F$ is a Finsler metric $C$-bi-Lipschitz equivalent to $F_0=\sqrt{g_0}$, then the unit ball of $F$ needs to contain the Euclidean sphere $S(C^{-1})$. Let $B^F(1) \subset T_xM$ be the unit sphere for $F$. By convexity of $B^F(1)$, the tangent hyperplane $\mathcal{H}_F$ to $B^F(1)$ at $\ell_{F}^{-1} \circ \ell_{F_0} (\theta_1, \dots, \theta_{n-1})$ needs to intersect $\mathcal{P}$ above $\mathcal{H}_0 \cap \mathcal{P}$. Otherwise, $B^F(1)$ would intersect $S(C^{-1})$.

Furthermore, as $B^F(1)$ needs to be contained in $S(C)$, the Euclidean sphere of center $0$ and radius $C$, if $X \geq C$, then the orthogonal projection of $\ell_{F}^{-1} \circ \ell_{F_0} (\theta_1, \dots, \theta_{n-1})$ on $\mathcal{H}_F \cap \mathcal{P}$ needs to be below the line $(0A)$. Hence, if $X \geq C$, we have
\begin{equation*}
 \left| \cos\left( \ell_{F}^{-1} \circ \ell_{F_0}(\theta_{n-1}) \right) \right| \geq \cos \psi.
\end{equation*}

Now, 
\begin{equation*}
 X^2 = \left(\frac{C^{-1}}{2 \sin \eps}\right)^2 + \left(\frac{C^{-1}\sin \eps}{2 }\right)^2 = \frac{C^{-2}}{4}\left(\sin^{-2}\eps + \sin^{2}\eps \right).
\end{equation*}
So, if we set $\eps$ such that 
\begin{equation*}
 \sin^2 \eps = 2 C^6 - \sqrt{4 C^{12}-1},
\end{equation*}
then $X = C$. As $\cos \psi = \cos \eps/ (2CX) = C^{-2}\cos \eps / 2$, we obtain
\begin{equation*}
 \left| \cos \left( \ell_{F}^{-1} \circ \ell_{F_0}(\theta_{n-1}) \right) \right| \geq \frac{C^{-2} \cos \eps}{2}\, . \qedhere
\end{equation*}
\end{proof}

Now that we proved the claim, the Proposition follows easily. Let $\eps>0$ be chosen as in the Claim, and
\[
U(\eps) := \lbrace (\theta_1, \dots, \theta_{n-1})  \mid  \cos \theta_{n-1} \geq \cos \eps \rbrace.
\]
We have
\begin{align*}
  \lVert p\rVert_{\sigma^F} &\geq c_n C^{-2n -2} \int_{H_xM} \left(L_{X_0} \pi^{\ast} \phi \circ \ell_{F}^{-1} \circ \ell_{F_0} \right)^2 \alpha^{F_0} \\
   &\geq c_n C^{-2n -2} \int_{U(\eps)}  \left( \cos\left( \ell_{F}^{-1} \circ \ell_{F_0}(\theta_{n-1}) \right) \right)^2 \alpha^{F_0} \\
   &\geq c_n C^{-2n -2}  \frac{C^{-4} \cos^2 \eps}{4} \int_{U(\eps)} \alpha^{F_0},
\end{align*}
and as $\eps$ depends only on $C$, $\int_{U(\eps)} \alpha^{F_0}$ depends only on $C$ and $n$. Therefore, there exists a constant $K(C,n)$ such that
\begin{equation*}
 \lVert p\rVert_{\sigma^F} \geq C^{-K(C,n)}.
\end{equation*}
This finishes the proof of the Proposition.
\end{proof}

In order to prove Theorem \ref{thm_energy_control}, we use the following result
\begin{thm}[Matveev, Troyanov \cite{MatveevTroyanov}] \label{thm_Binet_Legendre}
 Let $F$ be a $c$-quasireversible Finsler metric. There exists a Riemannian metric $g_F$, called the associated Binet--Legendre metric, with the following properties:
\begin{enumerate}
 \item The metric $g_F$ is as smooth as $F$;
 \item The metrics $F$ and $g_F$ are bilipschitz-equivalent. More precisely, if $n$ is the dimension of $M$, then
  \begin{equation*}
   (c\sqrt{2 n})^{-n-1} \sqrt{g_F} \leq F \leq (c\sqrt{2 n})^{n+1} \sqrt{g_F}.
  \end{equation*}
\end{enumerate}
\end{thm}

\begin{proof}[Proof of Theorem \ref{thm_energy_control}]
 Let $g_F$ and $g_{F_0}$ be the Binet--Legendre metrics associated with $F$ and $F_0$. Let $\sigma_F$ and $\sigma_{F_0}$ be the symbols of $\Delta^F$ and $\Delta^{F_0}$. We can write
\begin{equation*}
 \frac{\sigma_F}{\sigma_{F_0}}= \frac{\sigma_F}{g_F^{\ast}} \frac{g_F^{\ast}}{\left(F^{\ast}\right)^2} \frac{\left(F^{\ast}\right)^2}{\left(F_0^{\ast}\right)^2} \frac{\left(F_0^{\ast}\right)^2}{g_{F_0}^{\ast}} \frac{g_{F_0}^{\ast}}{\sigma_{F_0}} \, .
\end{equation*}
Now, if $C_1$ and $C_2$ are the quasireversibility constants of $F$ and $F_0$ respectively, then, by Proposition \ref{prop_symbol_control} and Theorem \ref{thm_Binet_Legendre}, there exists $K_1 = K_1(C_1,n)$ and $K_2 = K_2(C_2,n)$ such that:
\begin{align*}
(C_1\sqrt{2 n})^{-K_1} &\leq \frac{\sigma_F}{g_F^{\ast}} \leq (C_1\sqrt{2 n})^{K_1} \\
(C_2\sqrt{2 n})^{-K_2} &\leq \frac{\sigma_{F_0}}{g_{F_0}^{\ast}} \leq (C_2\sqrt{2 n})^{K_2}.
\end{align*}
So,
\begin{equation*}
 (C_1\sqrt{2 n})^{-K_1-2n-2} (C_2\sqrt{2 n})^{-K_2-2n-2} C^{-2} \leq  \frac{\sigma_F}{\sigma_{F_0}} \leq C^2 (C_1\sqrt{2 n})^{K_1+2n+2} (C_2\sqrt{2 n})^{K_2+2n+2}.
\end{equation*}

From the above equation together with Equation \eqref{eq_control_volume}, we immediately deduce the Theorem \ref{thm_energy_control}. Indeed,
\begin{align*}
 E^F(f) &= c_n \int_{x\in M} \left(\int_{H_xM} \left( L_X\pi^{\ast} f \right)^2 \alpha^F \right) \Omega^F \\
    &= c_n \int_{x\in M} \lVert df \rVert^2_{\sigma_F} \Omega^F \\
    &\leq c_n C^2 (C_1\sqrt{2 n})^{K_1+2n+2} (C_2\sqrt{2 n})^{K_2+2n+2} C^n \int_{x\in M} \lVert df \rVert^2_{\sigma_{F_0}} \Omega^{F_0} \\
    &\leq C^{n+2} (C_1\sqrt{2 n})^{K_1+2n+2} (C_2\sqrt{2 n})^{K_2+2n+2} E^{F_0}(f),
\end{align*}
and similarly, we have
\begin{equation*}
 E^F(f) \geq C^{-n-2} (C_1\sqrt{2 n})^{-K_1+2n+2} (C_2\sqrt{2 n})^{-K_2+2n+2} E^{F_0}(f). \qedhere
\end{equation*}
\end{proof}

\section{Applications} \label{sec_applications}

In this section, we apply Theorem \ref{thm_energy_control} to some of what is known in the Riemannian context using the Binet--Legendre metric. The moral of this section being that any coarse control of the spectrum for a class of Riemannian metric immediately gives a coarse control of the Finsler spectrum thanks to Theorem \ref{thm_energy_control}. We did not try to give an exhaustive list of such applications, but just concentrated on some famous Riemannian problems.

\subsection{Spectral control on surfaces}

\begin{thm}\label{thm_control_on_surfaces}
 Let $\Sigma$ be a surface of genus $\delta$ and $C_1 \geq 1$. There exists a constant $K \geq 1$, depending only on $C_1$, such that, for any $C_1$-quasireversible Finsler metric $F$ on $\Sigma$ and for any $k\in \N$,
\begin{equation*}
 \lambda_k(\Sigma,F) \left(\vol(\Sigma,F) \right) \leq (2C_1)^K (1+\delta) k.
\end{equation*}
\end{thm}

We will show in Section \ref{sec_big_eigenvalues} that the dependency on the quasireversibility constant is essential. Indeed, we construct some examples of Finsler metrics on surfaces, highly non-reversible, with arbitrarily large eigenvalues.

To prove this result, we just use the Binet-Legendre metric and what is known about the eigenvalues of Riemannian surfaces and apply Corollary \ref{cor_control_spectrum}.

\begin{proof}
Let $F$ be a $C_1$-quasireversible Finsler metric on $\Sigma$ and $g_{F}$ its Binet-Legendre metric. Korevaar \cite{Korevaar:Upper_bounds} proved that there exists a universal constant $c$, such that, for any $k$, the $k$-th eigenvalue of $g_{F}$ verifies:
\begin{equation*}
 \lambda_k(g_F) \left(\vol(\Sigma,g_F) \right) \leq c (1+\delta) k.
\end{equation*}
Now as $F$ and $g_F$ are bi-Lipschitz equivalent, by Corollary \ref{cor_control_spectrum}, there exists a constant $K\geq 1$, depending only on $C_1$, such that
\begin{equation*}
 (2C_1)^{-K} \leq \frac{\lambda_k(F)}{\lambda_k(g_F)} \leq (2C_1)^{K}.
\end{equation*}
And, by Lemma \ref{lem_control_of_m_mu},
\begin{equation*}
 (2C_1)^{-6} \leq \frac{\vol(\Sigma,F)}{\vol(\Sigma,g_F)} \leq (2C_1)^{6}.
\end{equation*}
So,
\begin{equation*}
 \lambda_k(F) \leq (2C_1)^K \lambda_k(g_F) \leq (2C_1)^K \frac{c(1+\delta)k}{\vol(\Sigma,g_F)} \leq (2 C_1)^{K+6} \frac{c (1+\delta)k}{\vol(\Sigma,F)} \, . \qedhere
\end{equation*}
\end{proof}

\subsection{Spectral control in a conformal class}

\begin{thm} \label{thm_conformal_control}
Let $M$ be an $n$-manifold and $F$ a Finsler metric on $M$. There exists a constant $C_n\left([F] \right)$, depending only on the dimension $n$ and the conformal class of $F$, such that, for any $k\in \N$, 
\begin{equation*}
 \lambda_k(M,F) \left(\vol\left(M,F \right) \right)^{2/n} \leq C_n\left( [F]\right) k^{2/n}.
\end{equation*}
\end{thm}

\begin{proof}
Let $F$ be a Finsler metric and $g_F$ its associated Binet--Legendre metric. By Korevaar's Theorem \cite{Korevaar:Upper_bounds}, there exists a constant $C_n \left([g_F] \right)$, depending only on $n$ and the conformal class of $g_F$, such that 
\begin{equation*}
 \lambda_k(M,g_F) \left(\vol\left(M,g_F \right) \right)^{2/n} \leq C_n\left( [g_F]\right) k^{2/n}.
\end{equation*}
Using Corollary \ref{cor_control_spectrum} we obtain that for some constant $K$, depending only on $n$ and the quasireversibility constant of $F$, we have
\begin{equation*}
 \lambda_k(M,F) \left(\vol\left(M,F \right) \right)^{2/n} \leq K C_n\left( [g_F]\right) k^{2/n}.
\end{equation*}
Now, when $F_1$ and $F_2$ are in the same conformal class, then $g_{F_1}$ and $g_{F_2}$ are also in the same conformal class (see \cite{MatveevTroyanov}), so the constant $K C_n\left([g_F]\right)$ depends on $n$ and the conformal class of $F$.
\end{proof}

\subsection{Small eigenvalues in a tower of coverings}

Our last application gives a condition for when it is possible to construct small eigenvalues on coverings of a given Finsler manifold. This is a direct application of a Theorem of T.\ Mantuano \cite{Mantuano:discretization_apply_to_spectrum}, which itself generalized a result by R.\ Brooks \cite{Brooks:spectral_tower_coverings}.

\begin{thm} \label{thm_spectrum_covering}
 Let $M$ be an $n$-manifold and $F$ a Finsler metric on $M$. Let $\{M_i\}_{i\geq 1}$ be a family of finite-sheeted covering spaces of $M$ with their induced Finsler metric. Let $\Gamma_i$ be the Schreier graph of the subgroup $\pi_1(M_i)$ of $\pi_1(M)$. Then, there exists a constant $C \geq 1$ depending on $n$ and $F$ such that, for all $k < \left|\Gamma_i \right| $
\begin{equation*}
 C^{-1} \lambda_k \left( \Gamma_i \right)  \leq \lambda_k\left(M_i \right) \leq C \lambda_k \left( \Gamma_i \right).
\end{equation*}
In particular, for all $k$
\begin{equation*}
 \lambda_k\left(M_i \right) \rightarrow 0 \; \text{when } i\rightarrow +\infty \Leftrightarrow \lambda_k \left( \Gamma_i \right) \rightarrow 0 \; \text{when } i\rightarrow +\infty .
\end{equation*}
\end{thm}

\begin{proof}
 It suffices to apply Theorem I.4.1 of \cite{Mantuano:discretization_apply_to_spectrum} to the Binet--Legendre metric associated to $F$.
\end{proof}

Among the corollaries of this result, we have that (see \cite{Brooks:spectral_tower_coverings}):
\begin{itemize}
 \item If $\pi_1(M)$ is infinite, amenable and residually finite, then there exists a sequence of finite coverings $\{M_i\}_{i\geq 1}$ such that $\lambda_k\left(M_i \right) \rightarrow 0 $;
 \item If $\pi_1(M)$ has Kazdhan's property (T), then there exists a constant $C>0$ such that $\lambda_1(M_i) > C$ for any covering $M_i$ of $M$.
\end{itemize}

\subsection{Remark about weighted Laplacians}

The Finsler--Laplacian is a \emph{weighted Laplacian} (see \cite{moi:natural_finsler_laplace}), i.e., if $F$ is a Finsler metric, $\sigma_F$ is the symbol of the Finsler--Laplacian $\Delta^F$ and $a \colon M \rightarrow \R^+$ is the function defined by $\Omega^F = a \Omega^{\sigma_F}$, where $\Omega^{\sigma_F}$ is the volume element of the Riemannian metric dual to $\sigma_F$, then
\begin{equation*}
 \Delta^{F}  = \Delta^{\sigma_F}  - \frac{1}{a} \langle \nabla \cdot , \nabla a \rangle.
\end{equation*}

We can see that Theorem \ref{thm_spectrum_covering} can be easily deduce from the characterization of the Finsler--Laplacian as a weighted Laplacian. Indeed, for any $f \in H^1(M)$,
\begin{equation*}
 E^F(f) = \int_M \lVert df \rVert_{\sigma_F}^2 \Omega^F \leq \sup_{x\in M} a(x) \int_M \lVert df \rVert_{\sigma_F}^2 \Omega^{\sigma_F}.
\end{equation*}
So, by the Min-Max principle,
\begin{equation*}
 \lambda_k(F) \leq \frac{\sup_{x\in M} a(x)}{\inf_{x\in M} a(x)} \lambda_k(\sigma_F) \leq C \lambda_k(\sigma_F),
\end{equation*}
where $C>0$ is a constant depending on $F$. and we obtain as easily that $\lambda_k(F) \geq C^{-1} \lambda_k(\sigma_F)$. Hence, applying Mantuano's result to $\sigma_F$ gives Theorem \ref{thm_spectrum_covering}.

Similarly, but with a tiny bit more work, the weighted Laplacian characterization could also be used to deduce Theorem \ref{thm_conformal_control}. The only thing that one needs to prove is the following:
\begin{lem}
 Let $F$ be a Finsler metric on a $n$-manifold $M$, $f \colon M \rightarrow \R$ a smooth function, and $F_f := e^f F$. Then, 
\begin{align*}
 \Omega^{F_f} &= e^{nf} \Omega^F \\
 \sigma_{F_f} &= e^{-2f} \sigma_{F}.
\end{align*}
\end{lem}
With the definition of our Finsler--Laplace operator, the proof of this lemma is an exercise. It is done in \cite[Section 2.3]{moi:these}, although the second equation is not written explicitly.

However, Theorem \ref{thm_control_on_surfaces} \emph{cannot} be deduced as easily from that. Indeed, we would need a control of the density $a$ depending \emph{only} on the quasireversibility constant of the Finsler metric $F$.

Finally, note also that Theorem \ref{thm_energy_control} stays true for non-compact manifolds, whether the control of the energy via the weighted Laplacian that we gave above may not.

\section{Examples of large eigenvalues} \label{sec_big_eigenvalues}

In this section, for any given surface, we are going to construct Finsler metrics with an arbitrarily large first eigenvalue. Our examples are \emph{Randers metrics} and we start by giving some of their properties.

\subsection{Randers metrics}

A Randers metric $F$ on a $n$-manifold $M$ is given by
\[
 F = \sqrt{g_0} + \rho,
\]
where $g_0$ is a Riemannian metric and $\rho$ a $1$-form on $M$ such that $\lVert \rho \rVert_{g_0^{\ast}} < 1$. Under this condition, $F$ is a Finsler metric (see for instance \cite{BaoChernShen}).

We denote by $A$, $X$, $\alpha^F$ and $\Omega^F$ the Hilbert form, geodesic flow, angle form and volume form associated with $F$ and by $A_0$, $X_0$, $\alpha_0$ and $\Omega_0$ the same objects associated with $g_0$.

The study of the Finsler--Laplacian on Randers metrics is particularly agreeable because the objects associated with a Randers metric can be easily written in function of the Riemannian ones.
\begin{prop}
  We have the following equalities:
\begin{align}
 A &= A_0 + \pi^{\ast} \rho, \label{eq_A_randers}\\
 \ada &= \left(1+ \pi^{\ast} \rho (X_0) \right) A_0 \wedge dA_0^{n-1}, \label{eq_ada_randers}\\ 
 X &= \frac{1}{1+ \pi^{\ast} \rho (X_0)} X_0 + Y_0, \quad \text{where } Y_0 \in VHM,  \label{eq_X_randers}\\
 \Omega^F &= \Omega_0, \label{eq_omega_randers}\\
 \alpha^F &= \left( 1+ \pi^{\ast} \rho (X_0) \right) \alpha_0. \label{eq_alpha_randers}
\end{align}
\end{prop}
The proof of this result is in \cite[Chapter 3]{moi:these}, but we give it below for the convenience of the reader.

Note that the fact that the Holmes-Thompson volume of a Randers metric is equal to the volume of its Riemannian part is not new (see for instance \cite{ChengShen}), but not that widely known.

\begin{proof}
  By definition of $A$, for any $\xi \in H_{x}M$ and ${Z \in T_{\xi}HM}$, we have
\begin{align*}
 A_{\xi}(Z) &= \lim_{\eps \rightarrow 0} \frac{F_0\left(x,v + \eps d\pi (Z)\right) -F_0\left(x,v \right) + \eps \rho_x\left( d\pi (Z) \right)}{\eps} \\ 
  &= \left. A_0 \right._{\xi}(Z) + \pi^{\ast}\rho(Z).
\end{align*}
So we have Equation \eqref{eq_A_randers}.

From now on we will write $\rho$ instead of $\pi^{\ast}\rho$ as it will simplify notations and hopefully not lead to any confusion. Using this notation, we have: $A = A_0 + \rho$ and therefore $dA = dA_0 + d\rho$.\\
Note that $dA^{n-1} = dA_0^{n-1} + T$ where $T$ is a $(2n-2)$-form. So, as $d\rho$ is a $2$-form vanishing on $VHM$, and for $Y_1, Y_2 \in VHM$, $i_{Y_1} i_{Y_2} dA_0 = 0$, $T$ can be given at most $n-2$ vertical vectors, i.e., if $Y_1, \dots, Y_{n-1} \in VHM$, then $i_{Y_1} \dots i_{Y_{n-1}} T = 0$. Now this implies that the top-form $A\wedge T$ vanishes, hence $\ada = (A_0 + \rho)\wedge dA_0^{n-1}$. As $\ada$ and $A_0 \wedge dA_0^{n-1}$ are both volume forms, there exists a function $\lambda$ such that $\ada = \lambda A_0 \wedge dA_0^{n-1}$. We have 
$$
i_{X_0}(\ada) = (1+\rho(X_0)) dA_0^{n-1} = \lambda dA_0^{n-1},
$$
therefore $\lambda = 1+\rho(X_0)$ and we proved Equation \eqref{eq_ada_randers}.

There exists a function $m \colon HM \rightarrow \R$ and a vertical vector field $Y_0$ such that $X = m X_0 + Y_0$ (see \cite{Fou:EquaDiff}). So
\begin{equation*}
 1 = A(X) = A_0(m X_0 + Y_0) + \rho(m X_0 + Y_0) = m\left( A_0(X_0) + \rho(X_0) \right) = m \left(1+ \rho(X_0) \right),
\end{equation*}
which gives Equation \eqref{eq_X_randers}.

Let $\alpha^{\Omega_0}$ be defined by $\alpha^{\Omega_0} \wedge \pi^{\ast}\Omega_0 = \ada$. We have $\alpha^{\Omega_0} \wedge \pi^{\ast}\Omega_0 = \lambda A_0\wedge dA_0^{n-1} = \lambda \alpha_0 \wedge \pi^{\ast}\Omega_0 $, hence $\alpha^{\Omega_0}$ and $\lambda \alpha_0$ coincide on $VHM$. It is then immediate (by construction of $\Omega^F$, see \cite[Section 2.1]{moi:natural_finsler_laplace}) that 
\begin{equation*}
 \Omega^F = \frac{\int_{H_xM} (1+\rho(X_0)) \alpha_0}{\voleucl\left(\S^{n-1}\right)}  \Omega_0\, .
\end{equation*}
As the metric $g$ is Riemannian, it is reversible, therefore $\int_{H_xM} \rho(X_0) \, \alpha_0$ must be zero. Hence $\Omega^F = \Omega_0$ and $\alpha^F = (1+\rho(X_0)) \alpha_0$. Which are Equations \eqref{eq_omega_randers} and \eqref{eq_alpha_randers}.
\end{proof}

Using the above equations, we can see that the energy of a Randers metric has a very nice expression:
\begin{prop}
 For $f\in H^1(M)$, we have
\begin{equation*}
 E^F(f) = \frac{n}{\voleucl\left(\S^{n-1} \right)} \int_{HM} \frac{\left(L_{X_0} \pi^{\ast} f \right)^2}{1+ \pi^{\ast} \rho (X_0) }  A_0 \wedge dA_0^{n-1}.
\end{equation*}
\end{prop}

\begin{proof}
The energy associated to $\Delta^F$ is given by
\begin{equation*}
 E^F(f) = \frac{n}{\voleucl\left(\S^{n-1} \right)} \int_{HM} \left(L_{X} \pi^{\ast} f \right)^2 \ada.
\end{equation*}
We saw that $X = \left(1+ \pi^{\ast} \rho (X_0)\right)^{-1} X_0 + Y_0$ with $Y_0$ a vertical vector field. Hence $L_{Y_0} \pi^{\ast} f =0$ (because $f$ is a function on $M$), so
\begin{equation*}
 L_{X} \pi^{\ast} f = \left(1+ \pi^{\ast} \rho (X_0)\right)^{-1} L_{X_0} \pi^{\ast} f.
\end{equation*}
Therefore,
\begin{align*}
 E^F(f) &= \frac{n}{\voleucl\left(\S^{n-1} \right)} \int_{HM} \left(1+ \pi^{\ast} \rho (X_0)\right)^{-2}\left(L_{X_0} \pi^{\ast} f \right)^2 \left(1+ \pi^{\ast} \rho (X_0) \right) A_0 \wedge dA_0^{n-1} \\
 &= \frac{n}{\voleucl\left(\S^{n-1} \right)} \int_{HM} \frac{\left(L_{X_0} \pi^{\ast} f \right)^2}{1+ \pi^{\ast} \rho (X_0) }  A_0 \wedge dA_0^{n-1}. \qedhere
\end{align*}
\end{proof}

The Proposition suggest a way of getting large eigenvalues. We can choose a $1$-form $\rho$ with a norm very close to $1$, so that there will be a direction in which $1+ \pi^{\ast} \rho (X_0)$ is very small. So functions such that their derivative in that direction is not too small will have a huge energy. Then, we can change the Riemannian metric so that functions that varies mostly in the direction of the kernel of $\rho$ will also have a large energy.
In the case of the torus, this strategy works perfectly well. It also works with a bit more care in the case of the sphere. Finally, we will reduce the problem of finding large eigenvalues on a surface to finding large eigenvalues for the sphere.

\subsection{Large eigenvalues on the $2$-Torus} \label{subsec_large_eigen_on_torus}

We start by constructing large eigenvalues on the torus. This step is not at all necessary, as it could be deduced from the general method that we describe in the next Section. However, we want to give it as it contains all the Finsler technicalities for constructing large eigenvalues and none of the spectral theoretic ones that we will need in the general case.

Let $\T^2 = \R^2/ \Z^2$ with standard coordinates $(x,y)$. Let $h>1$ and $r = h^{-1}$. Let $g_0$ be the Riemannian metric on $\T^2$ given in coordinates by:
\begin{equation*}
 g_0:= \left[ \begin{array}{cc}
             h^2 & 0 \\
	     0 & r^2
            \end{array}\right].
\end{equation*}
Let $0\leq \eta<1$. We set $\rho := \eta h dx$. The norm of $\rho$ for $g_0^{\ast}$ is
\begin{equation*}
\lVert \rho \rVert_{g_0^{\ast}} = \sqrt{h^{-2}\left(\eta h\right)^2} = \eta.
\end{equation*}

So we can set $F := \sqrt{g_0} + \rho$ and $F$ is a Randers metric. Furthermore, the volume of $\T^2$ for $F$ is $1$.

\begin{thm}
 For $\eta$ sufficiently close to $1$, we have
\begin{equation*}
 \lambda_1(F) \geq \frac{4\pi^2}{r^2} \, .
\end{equation*}
\end{thm}

\begin{proof}
 
Let $(x,y,\xi_x,\xi_y)$ be the standard coordinates on $T\T^2$. Let $\theta$ be the coordinate on $H_{x,y}\T^2$, defined by
\begin{align*}
 \cos \theta &= \frac{h\xi_x}{\sqrt{h^2 \xi_x^2 + r^2 \xi_y^2}} \\
 \sin \theta &= \frac{h\xi_y}{\sqrt{h^2 \xi_x^2 + r^{2} \xi_y^2}}.
\end{align*}
By definition, we have
\begin{equation*}
 A_0 = \frac{\partial \sqrt{g_0}}{\partial \xi_x} dx  + \frac{\partial \sqrt{g_0}}{\partial \xi_y}  dy = h \cos \theta dx + r \sin \theta dy.
\end{equation*}
A simple verification shows that
\begin{equation*}
 X_0 = h^{-1} \cos \theta \frac{\partial}{\partial x} + r^{-1} \sin \theta \frac{\partial}{\partial y}.
\end{equation*}
So, $\pi^{\ast}\rho (X_0) = \eta h h^{-1} \cos \theta   = \eta \cos \theta$.

Now we can start computing the energy of a function $f \in H^1(\T^2)$.
First, we have
\begin{equation*}
 L_{X_0} \pi^{\ast} f = h^{-1} \cos \theta \frac{\partial f}{\partial x} + r^{-1} \sin \theta \frac{\partial f}{\partial y},
\end{equation*}
and
\begin{equation*}
 A_0\wedge dA_0 = - d\theta \wedge dx \wedge dy.
\end{equation*}
Hence,
\begin{multline*}
 E^F(f) = \frac{1}{\pi} \int_{0\leq x , y \leq 1}\Bigg[ \int_{0}^{2\pi} \frac{h^{-2} \cos^2\theta}{ 1 + \eta \cos \theta} d\theta \left(\frac{\partial f}{\partial x} \right)^2  + \int_{0}^{2\pi} \frac{2\cos\theta \sin\theta}{ 1 + \eta \cos \theta} d\theta \left(\frac{\partial f}{\partial x}\frac{\partial f}{\partial y} \right)  \\ + \int_{0}^{2\pi} \frac{r^{-2} \sin^2\theta}{ 1 + \eta \cos \theta} d\theta \left(\frac{\partial f}{\partial y} \right)^2 \Bigg] dxdy.
\end{multline*}
Computations (using Mathematica) gives
\begin{align} 
 \int_{0}^{2\pi} \frac{h^{-2} \cos^2\theta}{ 1 + \eta \cos \theta} d\theta&= \frac{2\pi h^{-2} }{\left( 1 + \sqrt{1 - \eta^2} \right) \sqrt{1 - \eta^2} } \label{eq_calcul_integral_1}\\
 \int_{0}^{2\pi} \frac{2\cos\theta \sin\theta}{ 1 + \eta \cos \theta} d\theta&= 0  \\
 \int_{0}^{2\pi} \frac{r^{-2} \sin^2\theta}{ 1 + \eta \cos \theta} d\theta &=  \frac{2\pi r^{-2}}{\left( 1 + \sqrt{1 - \eta^2} \right)} \label{eq_calcul_integral_2}.
\end{align}
So we have
\begin{multline*}
  E^F(f) =  \frac{2 h^{-2} }{\left( 1 + \sqrt{1 - \eta^2} \right) \sqrt{1 - \eta^2} }  \int_{0\leq x , y \leq 1} \left(\frac{\partial f}{\partial x} \right)^2 dxdy \\ + \frac{2 r^{-2}}{\left( 1 + \sqrt{1 - \eta^2} \right)} \int_{0\leq x , y \leq 1} \left(\frac{\partial f}{\partial y}\right)^2 dxdy . 
\end{multline*}

For any fixed $h$, we see that
\begin{equation*}
 \lim_{\eta \rightarrow 1} \frac{2 h^{-2} }{\left( 1 + \sqrt{1 - \eta^2} \right) \sqrt{1 - \eta^2} } =  + \infty,
\end{equation*}
and, for any $0\leq \eta <1$,
\begin{equation*}
 \frac{2 r^{-2}}{\left( 1 + \sqrt{1 - \eta^2} \right)} \geq r^{-2}.
\end{equation*}

Let us choose $\eta$ such that $ 2 h^{-2} \left( \left( 1 + \sqrt{1 - \eta^2} \right) \sqrt{1 - \eta^2}\right)^{-1} \geq r^{-2}$. Then
\begin{equation*}
 E^F(f) \geq r^{-2} \int_{0\leq x , y \leq 1} \left(\frac{\partial f}{\partial x} \right)^2 + \left(\frac{\partial f}{\partial y}\right)^2 dxdy,
\end{equation*}
and the Rayleigh quotient of $f$ is such that 
\begin{equation*}
 R^F(f) \geq r^{-2} \frac{\int_{0\leq x , y \leq 1} \left(\frac{\partial f}{\partial x} \right)^2 + \left(\frac{\partial f}{\partial y}\right)^2 dxdy}{\int_{0\leq x , y \leq 1} f^2 dxdy}.
\end{equation*}
Therefore, by the Min-Max principle,
\[
 \lambda_1(F) \geq 4\pi^2 r^{-2}. \qedhere
\]
\end{proof}

\begin{rem}
Note that, in this case, the Finsler--Laplace operator $\Delta^F$ is equal to the Laplace--Beltrami operator of $\sigma_F$, the symbol of $\Delta^F$. Indeed, $\Delta^F$ is a weighted Laplacian where the weight is just a constant. However, this is not in contradiction with what is known about the bounds of eigenvalues on torus. Even so the volume of the torus for the Randers metric is equal to $1$, the volume of the torus for the symbol metric tends to zero as $\eta$ tends to $1$.
\end{rem}

\subsection{Large eigenvalues on surfaces}

We want to show the following
\begin{thm}\label{thm_big_eigenvalues}
 For any surface $\Sigma$ and any $M>0$, there exists a Randers metric $F = \sqrt{g} + \rho$ such that 
\begin{equation*}
 \lambda_1(F) \vol\left(\Sigma, F \right) \geq M.
\end{equation*}
\end{thm}

The proof will consist of adding a very long Randers ``nose'' to the surface $\Sigma$, and showing that if there was a topological bound on the eigenvalue of $F$ on $\Sigma$, then there would be a bound on the eigenvalue of Randers metrics on the sphere obtained by removing the topology of $\Sigma$. On the sphere, we are able to use the same kind of Randers metric as on the torus above and prove that its eigenvalues are large.
Let us first describe what we mean by this Randers nose.

Let $S_0$ be a sphere embedded in $\R^3$ in such a way that it is symmetric with respect to the $(0,y,z)$-plan and invariant by rotation around the $x$-axis. We also suppose that $S_0$ can be decomposed into the following pieces. We just describe the part of $S_0$ in the $x\geq 0$ half-space, from left to right (the other side being symmetric) (see Figure \ref{fig_la_sphere}).
\begin{figure}[t]
\begin{pspicture}(-5,-5)(5,5)
\psset{viewpoint=20 -40 40,Decran=40}
 \axesIIID[showOrigin=false](0,0,0)(15,4,4)



\psPoint(0,-1,0){X0}
\psPoint(7.5,-1,0){X1}
\psPoint(7.5,1,1){Z}
\psPoint(8,-1,0){X2}
\psPoint(8.2,-2.2,0){Y2}
\psPoint(10,-1,0){X4}
\psPoint(13,1.5,2){Z1}

\uput[l](Z1){$\color{OliveGreen} S_k$}
\psbrace[rot=90,ref=t](X0)(X1){$\color{red} C_L$}
\psbrace[rot=90,ref=t](X1)(X4){${\color{blue} C_1} \cup {\color{violet} C_2 } \cup {\color{green} C_3}$}

\defFunction[algebraic]{sphere1}(x,t){x+11.73205081}{(4-x^2)^(1/2)*cos(t)}{(4-x^2)^(1/2)*sin(t)}
\psSolid[object=surfaceparametree,
    opacity=0.7,
    linecolor={OliveGreen},
    base= -1.73205081 2 0 6.28318531,
    action = draw,
    function=sphere1,
    linewidth=0.5\pslinewidth,ngrid=20 20]

\defFunction[algebraic]{sphere2}(x,t){x-11.73205081}{(4-x^2)^(1/2)*cos(t)}{(4-x^2)^(1/2)*sin(t)}
\psSolid[object=surfaceparametree,
    opacity=0.7,
    linecolor=black,
    base= -2 1.73205081 0 6.28318531,
    action = draw,
    function=sphere2,
    linewidth=0.5\pslinewidth,ngrid=20 20]

\psSolid[object=cylindrecreux,
    h=7.5,
    r=1,
    RotY=90,
    action=draw,
    linewidth=0.01,
    ngrid=10 25,
    opacity = 0.7,
    linecolor=red,
    incolor=white](0,0,0)

\psSolid[object=cylindrecreux,
    h=0.5,
    r=1,
    RotY=90,
    action=draw,
    linewidth=0.01,
    ngrid=10 25,
    opacity = 0.7,
    linecolor=blue,
    incolor=white](7.5,0,0)
\psSolid[object=cylindrecreux,
    h=1,
    r=1,
    RotY=90,
    action=draw,
    linewidth=0.01,
    ngrid=10 25,
    opacity = 0.7,
    linecolor=violet,
    incolor=white](8,0,0)

\psSolid[object=cylindrecreux,
    h=1,
    r=1,
    RotY=90,
    action=draw,
    linewidth=0.01,
    ngrid=10 25,
    opacity = 0.7,
    linecolor=green,
    incolor=white](9,0,0)

\psSolid[object=cylindrecreux,
    h=10,
    r=1,
    RotY=-90,
    action=draw,
    linewidth=0.01,
    ngrid=10 25,
    opacity = 0.7,
    fillcolor=red,
    incolor=white](0,0,0)
\end{pspicture}
\caption{} \label{fig_la_sphere}
\end{figure}
\begin{itemize}
 \item A cylinder $C_L$ of length $L>1$ and radius $1$;
 \item A cylinder $C_1$ of small length $l_1$ and radius $1$;
 \item Two cylinders $C_2$ and $C_3$ each of length $1$ and radius $1$;
 \item A sphere $S_k$ of radius $k\geq 1$, minus a disk of radius $1$.
\end{itemize}
We write $C^-_L$, $C^-_1$, $C^-_2$, $C^-_3$ and $S^-_k$ for the symmetric parts.

Now we construct $\Sigma$, embedded in $\R^3$, such that $\Sigma$ and $S_0$ coincide apart from on $S_k$. We denote by $S$ the part of $\Sigma$ which is different from $S_0$.

We equip the surfaces $S_0$ and $\Sigma$ with the Riemannian metrics obtained from $\R^3$. We suppose that $S$ was normalized so that it has the same volume as $S_k$, and we suppose that $k$ is large enough so that the injectivity radius of $\Sigma$ is greater than $1$.

From now on, everything is fixed apart from $L$, which can be taken as large as we want.

Now, for $0\leq \eta <1$, we define the $1$-form $\rho$ on $C_1^- \cup C^-_L \cup C_L \cup C_1$ by:
\begin{equation*}
 \rho_p = \begin{cases}
          \eta dx \quad \text{when } p \in C^-_L \cup C_L \\
	  \eta f(x) dx \quad \text{when } p \in C_1 \\
	  \eta f(-x) dx \quad \text{when } p \in C^-_1,
         \end{cases}
\end{equation*}
where $f(x)$ is any decreasing smooth function with value $1$ at $x=L$ and $0$ at $x=L + l_1$. We extend $\rho$ to be vanishing everywhere else. Note that the norm of $\rho$ (considered as a $1$-form on $\Sigma$ or $S_0$) is strictly less than $1$.

We define a Randers metric on $\Sigma$ and on $S_0$ by $F = \sqrt{g} + \rho$.

The proof of Theorem \ref{thm_big_eigenvalues} will follow from the next two lemmas.

\begin{lem} \label{lem_construction_u_tilde}
 Let $L$ be large and $\lambda_1(\Sigma, F)$ the first eigenvalue of $-\Delta^F$ on $\Sigma$.
 If $\lambda_1(\Sigma, F) \leq L^{-4/5}$ then there exists a function $\tilde{u} \colon \Sigma \rightarrow \R$ such that:
\begin{itemize}
 \item The function $\tilde{u}$ is constant on $S$;
 \item The function $\tilde{u}$ is orthogonal to the constant functions on $\Sigma$;
 \item There exists a constant $K>0$, depending only on the geometry of the surface $S$ (and \emph{not} on $L$), such that the Rayleigh quotient of $\tilde{u}$ verifies
 \begin{equation*}
  R^F(\tilde{u}) \leq \frac{K}{L^{2/5}} \,.
 \end{equation*}
\end{itemize}
\end{lem}

Lemma \ref{lem_construction_u_tilde} tells us that an eigenfunction for an eigenvalue $\lambda_1 \leq  L^{-4/5}$ would be ``almost constant'' on the part $S$ of $\Sigma$. Hence, we can transplant this eigenfunction to a test function on the sphere $S_0$, because $S_0$ and $\Sigma$ coincides apart from on $S$ and $S_k$. To prove this result, we use a Sobolev inequality on $S$, which holds because the Finsler metric is just Riemannian on $S$. It is important to note that the geometry of $S$ is fixed, and does not depend on the variation of $L$ or of the Randers metric. This is why we get that the constant $K>0$ is independent of $L$.

\begin{lem} \label{lem_big_eigenvalue_on_sphere}
  Let $\lambda_1(S_0, F)$ be the first eigenvalue of $-\Delta^F$ on $S_0$. For large enough $L$ and $\eta$ close to $1$, we have
\begin{equation*}
 \lambda_1(S_0, F) > \frac{1}{L^{1/5}}.
\end{equation*}
\end{lem}

This Lemma tells us that, for a good choice of the $1$-form $\rho$ and of $L$, the first non-zero eigenvalue on the Randers sphere is large enough. We prove this thanks to two types of considerations. First, by symmetry, we see that an eigenfunction for $\lambda_1$ needs to be null in the middle of the cylinder, i.e., the intersection of $S_0$ and the $(y,z)$-plan. Then, two possibilities can occur: either the function grows enough on the cylinder or it does not. In the former case, thanks to purely Finslerian arguments of the exact same nature as in the torus case (Section \ref{subsec_large_eigen_on_torus}), we show that the Rayleigh quotient of the eigenfunction is large. In the later case, the restriction of the eigenfunction to $S_k$ needs to be close to a solution of the Dirichlet problem on $S_k$, and we once again show that the eigenvalue has to be large.

Let us assume these two lemmas, which are the bulk of the proof, and prove the Theorem \ref{thm_big_eigenvalues}.
\begin{proof}[Proof of Theorem \ref{thm_big_eigenvalues}]
 Suppose that $\lambda_1(\Sigma, F) \leq L^{-4/5}$ and $\tilde{u}$ is the function given by Lemma \ref{lem_construction_u_tilde}. As $\tilde{u}$ is constant on $S_k$, we can ``transpose'' $\tilde{u}$ to a function $v$ on $S_0$ so that $v$ and $\tilde{u}$ are equal, i.e., if $\tilde{u}(p) = a$ on $S_k$, we set
\begin{equation*}
 v(p) :=  \begin{cases}
        \tilde{u}(p), \text{ if } p \in \Sigma \smallsetminus S = S_0 \smallsetminus S_k \\
	a , \text{ if } p \in S.
       \end{cases}
\end{equation*}

 Outside $S$ and $S_k$, the Randers metrics on $\Sigma$ and $S_0$ take the same values, and as the volumes of $S$ and $S_k$ where chosen to be equal, we have
\begin{equation*}
 R^F(\tilde{u} ) = R^F(v).
\end{equation*}
So, by Lemma \ref{lem_construction_u_tilde}, $R^F(v) \leq K L^{-2/5}$ for some $K$ depending only on $S$. But $R^F(v)\geq \lambda_1(S_0,F)$, so Lemma \ref{lem_big_eigenvalue_on_sphere} gives a contradiction when $L$ is large.

Therefore, for $L$ large enough, $\lambda_1(\Sigma, F) \geq L^{-4/5}$, and as $\vol\left(\Sigma, F\right) = \vol\left(\Sigma, g\right) \geq L$, we finally obtain
\begin{equation*}
 \lambda_1(\Sigma, F) \vol\left(\Sigma, F\right) \geq L^{1/5} . \qedhere
\end{equation*}
\end{proof}

\subsubsection{Proof of Lemma \ref{lem_construction_u_tilde}}

Let $u$ be an eigenfunction of $-\Delta^F$ associated with $\lambda_1(\Sigma, F)$. We choose $u$ such that $\lVert  u \rVert_{L^2(\Sigma,F)} =1$.
Assume that $\lambda_1(S_1, F) < \frac{1}{L^{4/5}}$.

\begin{claim}
There exists a constant $K$ depending only on $(S,g)$ such that, for any $x,y \in S$,
\begin{equation*}
 \left|u(y) - u(x) \right| \leq K \sqrt{\lambda_1(\Sigma,F)} \leq \frac{K}{L^{2/5}}.
\end{equation*}
\end{claim}

\begin{proof}
 To prove the Claim, we follow the same method as in \cite[pp.136--137]{DodziukRandol}. First, we use the following Sobolev inequality (see \cite[6.22 and 6.29]{Warner_springer83}):
For any ball $B$ in $S$ of radius $1$ and and $x\in B$, there is some $K_1>0$ such that,
\begin{equation*}
 \left|du(x)\right| \leq K_1 \sum_{i=0}^{2} \lVert \Delta^i du \rVert_{L^2(B,g)},
\end{equation*}
where $\Delta$ is the Laplacian for $g$, which corresponds to $\Delta^F$ on $S$. Note that $K_1$ only depends on the geometry of $S$, which is fixed once and for all. As $\Delta^i$ and $d$ commute, for any $x \in B$, we have
\begin{equation*}
 \left|du(x)\right| \leq K_1 \sum_{i=0}^{2} \lambda_1(\Sigma,F)^i \lVert du \rVert_{L^2(B,g)} \leq K_2 \lVert du \rVert_{L^2(B,g)}
\end{equation*}

For any two points $x,y$ in $S$, there exists a path $\gamma$ of length controlled by the geometry of $(S,g)$. Furthermore, we can partition $\gamma$ in the following way. There exist $x_0 = x, \dots, x_k = y$ on $\gamma$ such that $\gamma \subset \bigcup_{i=0}^{k} B(x_i, 1/2)$, where $B(x_i, 1/2)$ is the ball of center $x_i$ and radius $1/2$, and such that any one of the balls intersects at most $m$ others. Note that the numbers $k$ and $m$ as well as the constant $K_2$ above depend only on the geometry of $(S,g)$ which we kept fixed from the beginning (see \cite[p.137]{DodziukRandol}).

\begin{multline*}
 \left| u(y) - u(x) \right| \leq \sum_{i=0}^{k} \left|u(x_{i+1}) - u(x_i)\right| \leq K_2 \sum_{i=0}^{k} \lVert du \rVert_{L^2(B(x_i,1/2))}  \\
 \leq K_2 k^{1/2} \left(\sum_{i=0}^{k} \lVert  du \rVert_{L^2(B(x_i,1/2))}^2 \right)^{1/2} \leq K_2 (km)^{1/2} \lVert  du \rVert_{L^2(M)}
  \\ \leq \sqrt{\lambda_1(S_1,F)} K_2 (km)^{1/2} \lVert  u \rVert_{L^2(M)}
\end{multline*}

Hence, there exists a constant $K_3 = K_2 (km)^{1/2} >0$, depending only on the geometry of $S$, such that, for any $x,y \in S$,
\begin{equation*}
 \left|u(y) - u(x) \right| \leq K_3 \sqrt{\lambda_1(S_1,F)} \leq \frac{K_3}{L^{2/5}}. \qedhere
\end{equation*}
\end{proof}

With the Claim, we can start constructing $\tilde{u}$.

Let $a = \max_{x\in S} u(x)$. For any $y\in S$, we have $\left|u(y) - a \right| \leq K_3/ L^{2/5}$. We define $H_0 := \Sigma \cap \lbrace x<0 \rbrace$.

Let $\chi$ be an increasing smooth function such that
\begin{equation*}
 \chi(x) = \begin{cases}
           0 \quad \text{si } x \in H_0 \cup C_L \cup C_1 \cup C_2 \\
	   1 \quad \text{si } x \in S,
          \end{cases}
\end{equation*}
and such that $\sup_{x\in C_3} \left|d\chi\right|^2 \leq 2$. Let
\begin{equation*}
 \bar{u} := \chi a + (1-\chi)u,
\end{equation*}
and
\begin{equation*}
 \tilde{u} := \frac{\bar{u} - \langle \bar{u}, \mathds{1} \rangle \mathds{1}}{\lVert \bar{u} - \langle \bar{u}, \mathds{1} \rangle \mathds{1} \rVert} \, ,
\end{equation*}
where $\langle \cdot , \cdot \rangle$ and $\lVert \cdot \rVert$ are the $L^2(\Sigma,F)$-scalar product and norm.

By construction $\tilde{u}$ is orthogonal to constant functions and constant on $S$. In order to prove the Lemma, we still have to control its Rayleigh quotient, which is
\begin{equation*}
 R^F(\tilde{u}) = E^F(\tilde{u}) = E^F\left(\bar{u} - \langle \bar{u}, \mathds{1} \rangle \mathds{1}\right) = E^F(\bar{u}).
\end{equation*}
As $\bar{u}$ and $u$ are equal on $H_0 \cup C_L \cup C_1 \cup C_2$, we just need to control the energy of $\bar{u}$ on $C_3 \cup S$, where the Randers metric is Riemannian.
\begin{align*}
 \left|E^F(\bar{u}) - E^F(u) \right|  &= \left|\int_{C_3 \cup S} \left|d \bar{u}\right|^2 - \left|du\right|^2 \right|\\
			&= \left| \int_{C_3 \cup S} \left|(a-u)d\chi +(1-\chi)du \right|^2 - \left|du\right|^2 \right|\\
			&= \begin{multlined}[t]%
			    \Biggl| \int_{C_3} (a-u)^2 \left|d\chi\right|^2 + 2(1-\chi)(a-u)\left|d\chi\right| \\  + \left((1-\chi)^2 -1\right) \left|du \right|^2  - \int_{S}\left|du\right|^2 \Biggr|
 			   \end{multlined} \\
			&\leq 4\pi \frac{K_3}{L^{4/5}} + 4 \sqrt{2\pi} \frac{\sqrt{K_3}}{L^{2/5}} + \frac{1}{L^{4/5}} \\
			&\leq \frac{K_4}{L^{2/5}}
\end{align*}

Therefore, we have
\begin{equation*}
R^F(\tilde{u}) \leq \frac{K_4}{L^{2/5}} + R^F(u) = \frac{K_4}{L^{2/5}} + \frac{1}{L^{4/5}} \leq \frac{K_5}{L^{2/5}} \,,
\end{equation*}
which ends the proof of Lemma \ref{lem_construction_u_tilde}.

\subsubsection{Proof of Lemma \ref{lem_big_eigenvalue_on_sphere}}

We want to show that $\lambda_1(\S_0,F) \geq L^{-1/5}$. 

Let us first express the energy on $(S_0,F)$. For any $f \in H^1(S_0)$, we have
\begin{align*}
 E^F(f) &= \frac{1}{\pi} \int_{HS_0} \frac{\left(L_{X_0} \pi^{\ast} f \right)^2}{1+ \pi^{\ast} \theta (X_0) }  A_0 \wedge dA_0^{n-1} \\
	&= \frac{1}{\pi} \Biggl(\begin{multlined}[t][.76\textwidth]
	                         \int_{H\left(C^-_L \cup C_L\right)} \frac{\left(L_{X_0} \pi^{\ast} f \right)^2}{1+ \pi^{\ast} \theta (X_0) }  A_0 \wedge dA_0^{n-1} \\ 
				  + \int_{H\left(S_0 \smallsetminus(C^-_L \cup C_L) \right)} \frac{\left(L_{X_0} \pi^{\ast} f \right)^2}{1+ \pi^{\ast} \theta (X_0) }  A_0 \wedge dA_0^{n-1}  \Biggr)
	                        \end{multlined} \\
	&\geq \frac{1}{\pi} \Biggl( \begin{multlined}[t][.76\textwidth]
				      \int_{H\left(C^-_L \cup C_L\right)} \frac{\left(L_{X_0} \pi^{\ast} f \right)^2}{1+ \pi^{\ast} \theta (X_0) }  A_0 \wedge dA_0^{n-1}  \\
				      + \frac{1}{2} \int_{H\left(S_0 \smallsetminus(C^-_L \cup C_L) \right)} \left(L_{X_0} \pi^{\ast} f \right)^2 A_0 \wedge dA_0^{n-1} \Biggr)
                                    \end{multlined} \\
	&\geq \frac{1}{\pi} \left( \int_{HC_{2h,r}} \frac{\left(L_{X_0} \pi^{\ast} f \right)^2}{1+ \pi^{\ast} \theta (X_0) }  A_0 \wedge dA_0^{n-1} \right) + \frac{1}{2} \int_{H\left(S_0 \smallsetminus(C^-_L \cup C_L) \right)} \lVert \nabla f \rVert^2 .
\end{align*}

Let us write $C_{2L}$ for $C^-_L \cup C_L$, we are going to compute the energy of a function on the cylinder $C_{2L}$. Let $(x,\theta ; \xi_{x}, \xi_{\theta})$ be coordinates on $TC_{2L}$. Here, $\theta$ is defined by 
\begin{align*}
 \cos\theta &= y \\
 \sin\theta &= z. \\
\end{align*}

On $C_{2L}$, the Riemannian metric is the standard metric given by $ds^2 = dx^2 + d\theta^2$. We define $\psi$, coordinate on $H C_{2L}$ by
\begin{equation*}
  \left\{ \begin{aligned}
           \cos \psi &= \frac{\xi_x}{\sqrt{\xi_x^2 + \xi_{\theta}^2}} \\
	   \sin \psi &= \frac{\xi_{\theta}}{\sqrt{\xi_x^2 + \xi_{\theta}^2}}.
         \end{aligned} \right.
\end{equation*}
In the coordinates $(x, \theta; \psi)$ on $HC_{2L}$, the Hilbert form is written as
\begin{equation*}
 A_{0, C_{2L} } = \cos \psi dx + \sin \psi d\theta,
\end{equation*}
and the geodesic flow is
\begin{equation*}
 X_{0, C_{2L} } = \cos \psi \frac{\partial}{ \partial x} + \sin \psi \frac{\partial}{ \partial \theta}\, .
\end{equation*}
So,
\begin{equation*}
 \rho\left(X_{0, C_{2L} }  \right) = \eta \cos \psi.
\end{equation*}

Using Equations \eqref{eq_calcul_integral_1} and \eqref{eq_calcul_integral_2}, we get that
\begin{align*}
 E^F_{C_{2L}}(f) &= \begin{multlined}[t][.7\textwidth]
		    \frac{1}{\pi} \int_0^{2\pi} \frac{\cos^2 \psi}{1 + \eta \cos \psi} d\psi \int_{-L}^L \int_0^{2\pi} \left( \frac{\partial f}{ \partial x} \right)^2  dx d\theta \\
                    + \frac{1}{\pi} \int_0^{2\pi} \frac{\sin^2 \psi}{1 + \eta \cos \psi}d\psi \int_{-L}^L \int_0^{2\pi} \left( \frac{\partial f}{ \partial \theta} \right)^2  dx d\theta 
		    \end{multlined}\\
		&=  \begin{multlined}[t][.7\textwidth]
		    \frac{2}{\left(1 + \sqrt{1- \eta^2}\right)\sqrt{1- \eta^2}} \int_{-L}^L \int_0^{2\pi} \left( \frac{\partial f}{ \partial x} \right)^2  dx d\theta \\ 
		    + \frac{2 }{\left(1 + \sqrt{1- \eta^2}\right)} \int_{-L}^L \int_0^{2\pi} \left( \frac{\partial f}{ \partial \theta} \right)^2  dx d\theta
		    \end{multlined}\\
		&\geq \begin{multlined}[t][.7\textwidth]
		    \frac{2}{\left(1 + \sqrt{1- \eta^2}\right)\sqrt{1- \eta^2}} \int_{-L}^L \int_0^{2\pi} \left( \frac{\partial f}{ \partial x} \right)^2  dx d\theta  \\ + \int_{-L}^L \int_0^{2\pi} \left( \frac{\partial f}{ \partial \theta} \right)^2 dx d\theta.
		      \end{multlined}
\end{align*}
Let
\begin{equation*}
 \frac{1}{\eps} := \frac{2}{\left(1 + \sqrt{1- \eta^2}\right)\sqrt{1- \eta^2}}.
\end{equation*}
Note that $\eps$ tends to $0$ as $\eta$ tends to $1$. We have
\begin{equation*}
 E^F_{C_{2L}}(f) \geq \eps^{-1} \int_{-L}^L \int_0^{2\pi} \left( \frac{\partial f}{ \partial x} \right)^2  dx d\theta +  \int_{-L}^L \int_0^{2\pi} \left( \frac{\partial f}{ \partial \theta} \right)^2 dx d\theta.
\end{equation*}

To prove Lemma \ref{lem_big_eigenvalue_on_sphere}, we will use the invariance of the metric (and hence of the Finsler--Laplacian) under rotation around the $x$-axis.

As $S^1$ acts by isometries, if $f$ is an eigenfunction of $-\Delta^F$, so is $\int f d\theta$. So there are two possible cases, either $\int f d\theta = 0$, or $\int f d\theta \neq 0$ and we deal with them separately.

\begin{claim} \label{claim_when_f_depends_on_theta}
 If $f$ is an eigenfunction of $-\Delta^F$, with eigenvalue $\lambda$, such that ${\int f d\theta =0}$, then
\begin{equation*}
 \lambda \geq \frac{1}{2}.
\end{equation*}
\end{claim}

\begin{proof}
 Let us write $f= v(\theta)u(x)$. The energy of $f$ verifies
\begin{align*}
 E^F(f) &\geq E^F_{C_{2L}}(f) + \frac{1}{2} \int_{S_0 \smallsetminus C_{2L}} \lVert \nabla f \rVert_g^2 \\
	&\geq \begin{multlined}[t]\eps^{-1} \int_{-L}^L \int_0^{2\pi} \left( \frac{\partial v(\theta)u(x)}{ \partial x} \right)^2 dx d\theta + \int_{-L}^L \int_0^{2\pi} \left( \frac{\partial v(\theta)u(x)}{ \partial \theta} \right)^2  dx d\theta \\ + \frac{1}{2} \int_{S_0 \smallsetminus C_{2L}} \lVert \nabla \left(v(\theta)u(x) \right) \rVert_g^2 
	      \end{multlined} \\
	&\geq \begin{multlined}[t] \int_0^{2\pi} v'(\theta)^2 d\theta \Biggl( \int_{-L}^L   u(x)^2 dx + \frac{1}{2}\biggl( \int_{L}^{L+l_1 + 2} u(x)^2 dx + \int_{-L-l_1 - 2}^{-L} u(x)^2 dx  \\ +  \int_{0}^{2\pi} u^2 \sin \phi d\phi \biggr) \Biggr).
	      \end{multlined}
\end{align*}
As $v$ is a $2\pi$-periodic function such that $\int v d\theta =0$, using the Fourier series of $v$, we see that
$$
\int_0^{2\pi} v'(\theta)^2 d\theta \geq \int_0^{2\pi} v(\theta)^2 d\theta.
$$
Hence,
\begin{equation*}
 E^F(f) \geq  \frac{1}{2} \int_{S_0} u^2v^2 \Omega^F .\qedhere
\end{equation*}
\end{proof}

Let $f$ be an eigenfunction for $\lambda_1(S_0,F)$.
If $f$ is such that $\int f d\theta = 0$, then the Claim proves the Lemma. From now on, we will suppose that $\int f d\theta \neq 0$ and set $u:= \int f d\theta$. The function $u$ is an eigenfunction for $\lambda_1(S_0,F)$ and depends only on the $x$-coordinate. Hence we have that
\begin{equation*}
 \lambda_1(S_0,F) = R^F(u) \geq \frac{\eps^{-1} \int_{C_{2L}} \lVert \nabla u \rVert_g^2 + \frac{1}{2} \int_{S_0 \smallsetminus C_{2L}} \lVert \nabla u \rVert_g^2 }{\int_{S_0} u^2 \Omega^F}.
\end{equation*}

\begin{claim}
 We can suppose that $u(0)=0$ and control the Rayleigh quotient of $u$ on $S_0^+ = S_0 \cap \lbrace x\geq 0 \rbrace$.
\end{claim}

\begin{proof}
 The Finsler--Laplace operator on $S_0$ is invariant under the involution $x \mapsto -x$, so the functions $v^+(x) = u(x) + u(-x)$ and $v^-(x) = u(x) - u(-x)$ are also eigenfunctions of $\Delta^F$. Either $v^-(x)$ is not zero and we can consider $v^-$ instead of $u$. Otherwise, $v^+ = u$. As $\int_{S_0} v^+ = 0 = \int_{S^+_0} v^+$, there exists an $x_0>0$ such that $v^+(x_0) = 0$. Let 
\begin{equation*}
\tilde{v}^+(x) = \begin{cases}
                  v^+(x)  \quad \text{if } x \geq x_0 \\
		  0   \quad \text{if } 0 \leq x \leq x_0.
                 \end{cases}
\end{equation*}
By definition, we have that $\lambda_1(S_0,F) = R^F(u) = R^F(v^+ ) \geq R^F(\tilde{v}^+)$. So giving a lower bound for $R^F(\tilde{v}^+)$ will give a lower bound for $\lambda_1(S_0,F)$.
\end{proof}

So now, our aim is to minimize 
\[
\frac{\eps^{-1} \int_{C_{L}} \lVert \nabla u \rVert^2 + \frac{1}{2} \int_{S_0^+ \smallsetminus C_L} \lVert \nabla u \rVert^2 }{\int_{S^+_0} u^2 \Omega^g} \,,
\]
where norms, gradient and integrals are taken with respect to the canonical Riemannian metric $g$ on $S_0$.

\begin{claim} \label{claim_energie_sur_cylindre}
 Let $a := u(L)$. We have
\begin{equation*}
 \int_{C_{L}} \lVert \nabla u \rVert^2  \geq \frac{2\pi a^2 }{ L} \,.
\end{equation*}
\end{claim}

\begin{proof}
We aim to minimize the functional $v \mapsto \int_{C_{L}} \lVert \nabla v \rVert_g^2$ for $v$ such that $v(0) = 0$ and $v(L) = a$.

A function $v$ is minimal if and only if, for any function $w$ such that $w(0)=w(L)=0$, 
\begin{equation*}
 \left. \frac{d}{dt} \right|_{t=0} \int_{C_{L}} \lVert \nabla \left(v + t w\right) \rVert^2 = 0.
\end{equation*}
Direct computation gives
\begin{align*}
 \left. \frac{d}{dt} \right|_{t=0} \int_{C_{L}} \lVert \nabla \left( v + t w \right)\rVert^2 &= \int_{C_{L}} 2 \langle \nabla v, \nabla w \rangle \\
  &= - 2 \int_{C_{L}} \langle \Delta v , w \rangle.
\end{align*}
Therefore $\int_{C_{L}} \lVert \nabla v \rVert^2$ is minimal when $\Delta v = 0$, that is, when
\begin{equation*}
v= \frac{ax}{L}\, .
\end{equation*}
Hence,
\begin{equation*}
  \int_{C_{L}} \lVert \nabla u \rVert^2  \geq \int_{C_{L}} \left\lVert \nabla \left(\frac{ax}{L}\right) \right\rVert^2 = \frac{a^2}{L^2} 2\pi L .\qedhere
\end{equation*}
\end{proof}

\begin{claim}\label{claim_norme_sur_hemisphere}
 Let again $a:= u(h)$. There exists two constants $K>0$ and $K'>0$, depending only on $k$, such that 
\begin{equation*}
 \int_{S^+_0 \smallsetminus C_L} u^2 \leq   k^2 \left( K \left( \int_{S^+_0 \smallsetminus C_L} \lVert \nabla u \rVert^2  \right) + a^2  K'  + 2a \sqrt{\frac{K'}{K}}   \left(\int_{S^+_0 \smallsetminus C_L} \lVert \nabla u \rVert^2 \right)^{1/2} \right).
\end{equation*}
\end{claim}

\begin{proof}
 As $u(h)=a$, the function $u - a$ restricted to $S^+_0 \smallsetminus C_L$ is a solution of the Dirichlet problem on $S^+_0 \smallsetminus C_L$. So, there exists some constant $K>0$ such that
\begin{equation*}
  \frac{K}{k^2} \int_{S^+_0 \smallsetminus C_L} (u-a)^2 \leq \int_{S^+_0 \smallsetminus C_L} \lVert \nabla (u-a) \rVert^2  = \int_{S^+_0 \smallsetminus C_L} \lVert \nabla u \rVert^2 .
\end{equation*}

Our aim is to control $\int_{S^+_0 \smallsetminus C_L} u^2 = \int_{S^+_0 \smallsetminus C_L}\left( (u- a )^2 + a^2 + 2a(u-a) \right)$. The only thing left to control is $\int_{S^+_0 \smallsetminus C_L} 2a (u-a)$. Let us choose $K'>0$, depending only on $k$ (which is fixed), such that $\vol\left(S^+_0 \smallsetminus C_L \right) \leq K' k^2$. We obtain that
\begin{align*}
 \int_{S^+_0 \smallsetminus C_L} 2a (u-a) &\leq 2a \left( \int_{S^+_0 \smallsetminus C_L} (u-a)^2 \right)^{1/2} \left( \int_{S^+_0 \smallsetminus C_L} 1 \right)^{1/2} \\
		     &\leq 2a \left(K' k^2\right)^{1/2} \left( \frac{k^2}{K} \int_{S^+_0 \smallsetminus C_L} \lVert \nabla u \rVert^2 \right)^{1/2} \\
		     &\leq 2a \left(\frac{K'}{K}\right)^{1/2} k^2  \left(\int_{S^+_0 \smallsetminus C_L} \lVert \nabla u \rVert^2 \right)^{1/2}.
\end{align*}
We deduce that
\begin{equation*}
 \int_{S^+_0 \smallsetminus C_L} u^2 \leq K k^2 \left( \int_{S^+_0 \smallsetminus C_L} \lVert \nabla u \rVert^2  \right) + a^2  K' k^2  + 2a \sqrt{\frac{K'}{K}} k^2  \left(\int_{S^+_0 \smallsetminus C_L} \lVert \nabla u \rVert^2 \right)^{1/2}. \qedhere
\end{equation*}
\end{proof}

We have now all the ingredients to finish the proof of Lemma \ref{lem_big_eigenvalue_on_sphere}. We suppose that $\int_{S^+_0} u^2 = 1$. Then, either $\int_{C_L} u^2 \geq 1/2$ or $\int_{S^+_0\smallsetminus C_L} u^2 \geq 1/2$.

If $\int_{C_L} u^2 \geq 1/2$, then
\begin{align*}
 \lambda_1(S_0,F) &\geq \frac{\eps^{-1} \int_{C_{L}} \lVert \nabla u \rVert^2 + \frac{1}{2} \int_{S^+_0\smallsetminus C_L} \lVert \nabla u \rVert^2 }{\int_{C_{L}} u^2 + \int_{S^+_0\smallsetminus C_L} u^2 } \\
	   &\geq \frac{\eps^{-1} \int_{C_{L}} \lVert \nabla u \rVert^2 }{2 \int_{C_{L}} u^2 } \\
	   &\geq \frac{\text{cste}}{\eps L^2},
\end{align*}
where the last inequality is obtained because the Rayleigh quotient of a non-constant function on a cylinder of length $L$ is greater than $\text{cste}/L^2$. Therefore, if we choose $\eps < (\text{cste})L^{-2}$, we get that $\lambda_1(S_0,F) \geq 1$. Note that we \emph{can} make such a choice for $\eps$ as it depends only on $\eta$.

So we proved that if $\int_{C_L} u^2 \geq 1/2$, then $\lambda_1(S_0,F) \geq 1 \geq L^{-1/5}$

Now suppose that $\int_{S^+_0\smallsetminus C_L} u^2 \geq 1/2$. Using Claim \ref{claim_energie_sur_cylindre}, we get
\begin{equation*}
 \lambda_1(S_0,F) \geq \frac{\eps^{-1} \int_{C_{L}} \lVert \nabla u \rVert^2 + \frac{1}{2} \int_{S^+_0\smallsetminus C_L} \lVert \nabla u \rVert^2 }{\int_{C_{L}} u^2 + \int_{S^+_0\smallsetminus C_L} u^2 } \geq \eps^{-1} \int_{C_{L}} \lVert \nabla u \rVert^2  \geq \frac{2\pi a^2}{\eps L} \,.
\end{equation*}
So that
\begin{equation*}
 a^2 \leq \frac{\lambda_1(S_0,F) \eps L}{2 \pi} \,.
\end{equation*}
Now, using the Claim \ref{claim_norme_sur_hemisphere}, we obtain
\begin{align*}
 \int_{S^+_0\smallsetminus C_L} u^2 &\leq k^2 \left( K \lambda_1(S_0,F) + a^2  K'  + 2a \left(\frac{K'}{K}\right)^{1/2}  \left(\lambda_1(S_0,F)\right)^{1/2} \right) \\
		&\leq \lambda_1(S_0,F) k^2 \left( K  + K' \frac{\eps L}{2 \pi} + \left(\frac{2 K'\eps L}{\pi K} \right)^{1/2} \right).
\end{align*}
We can choose $\eps$ small enough such that
\begin{equation*}
 \int_{S^+_0\smallsetminus C_L} u^2 \leq \lambda_1(S_0,F) k^2 \left( K + 1 \right).
\end{equation*}
Which yields, as $\int_{S^+_0\smallsetminus C_L} u^2 \geq 1/2$,
\begin{equation*}
 \lambda_1(S_0,F) \geq k^{-2}  \frac{K+1}{2}.
\end{equation*} 
Once again, for $L$ large enough, we have $\lambda_1(S_0,F) \geq L^{-1/5}$. This ends the proof of Lemma \ref{lem_big_eigenvalue_on_sphere}.

\bibliographystyle{amsplaineprint}
\bibliography{article_avec_bruno}

\end{document}